\documentclass[11pt]{article}
\usepackage{enumerate, fullpage}
\usepackage{graphicx}
\usepackage{wrapfig}
\usepackage{multirow}
\usepackage{algorithm}
\usepackage{algpseudocode}
\usepackage{color}
\usepackage{amsmath}
\usepackage{amsfonts}
\usepackage{amssymb}
\usepackage{mathrsfs}
\usepackage{eucal,bm}
\usepackage{lscape}
\usepackage{mathptmx}
\definecolor{darkgreen}{RGB}{0, 102, 0}
\definecolor{IKB}{RGB}{0, 47, 167}
\usepackage[pdftex,bookmarks,bookmarksopen=true,pdfstartview={FitH},bookmarksnumbered=true,colorlinks=true,citecolor=darkgreen,linkcolor=IKB,plainpages=false,pdfpagelabels]{hyperref}

\def\bR{\mathbf{R}}

\def\cF{{\cal F}}
\def\myB{{\cal B}}

\newcommand{\R}{\mathbb{R}}

\newcommand{\beq}{\begin{equation}}
\newcommand{\eeq}{\end{equation}}
 
\newcommand{\beqnr}{\begin{eqnarray}}
\newcommand{\eeqnr}{\end{eqnarray}}

\newcommand{\benum}{\begin{enumerate}}
\newcommand{\eenum}{\end{enumerate}}
\newcommand{\argmax}{\mathop{\rm argmax}}
\newcommand{\argmin}{\mathop{\rm argmin}}

\newcommand{\cA}{{\cal A}}

\newcommand{\cC}{{\cal C}}

\newcommand{\cI}{{\mathcal{I}}}

\newcommand{\cS}{{\cal S}}

\newcommand{\cZ}{{\cal Z}}

\newtheorem{DE}{Definition}[section]

\newtheorem{LE}[DE]{Lemma}
\newtheorem{EX}[DE]{Example}

\newtheorem{OB}[DE]{Observation}
\newtheorem{THM}[DE]{Theorem}
\newtheorem{CO}[DE]{Corollary}

\newcommand{\qed}{\mbox{}\hspace*{\fill}\nolinebreak\mbox{$\rule{0.7em}{0.7em}$}}

\newenvironment{PROOF}{\noindent {\sc Proof:}}{\(\qed\)\par}

\advance \topmargin by -\headheight
\advance \topmargin by -\headsep

\evensidemargin \oddsidemargin

\usepackage{prelim2e}

\begin{document}

\title{De-risking solutions to optimization problems\footnote{Funded by AFOSR and ONR.}}%\thanks{Supported by organization x.}}

\author{Daniel Bienstock and Blake Sisson, Columbia University}

%
\maketitle              % typeset the header of the contribution

\begin{abstract} 
  We develop a cutting-plane methodology that adjusts solutions to optimization problems so 
  as to reduce features that bring about exposure to risk, such as concentration of assets
  or resources.  The methodology is agnostic to the representation of risk but has provably good attributes. Our procedure aims
  to reduce the appropriate risk metric without accruing a significant increase in nominal 
  cost, rapidly, or proves that such an adjustment is not possible.  The underlying approach
  borrows from techniques used in first-order methods for optimization.
\end{abstract}

\section{Introduction}\label{sec:Introduction}
Consider a generic optimization problem 
\begin{align}
& \min \, c(x), \ \text{s.t.} \ x \in P. \label{eq:generic}
\end{align}
An empirical fact which is found in many real-world examples is that optimal or near-optimal solutions end up, inadvertently, incorporating high \emph{concentration}.  For example, in a logistical setting, a solution may place a large percentage of high-value items in a small set of locations during a narrow time span, causing exposure to many different types of actual risk. In a sense, concentration acts as an enabler for exogenous risk. This notion of concentration risk is well-known in
multiple industrial settings. Notably, a concrete algebraic representation of concentration suitable for
optimization may be very high dimensional.

In this paper we will assume that concentration (or, when appropriate,  risk) is approximately quantified through a function $\Phi(x)$.  To fix ideas
we will refer to $\Phi(x)$ as the \emph{impact} function.
The algorithms discussed in this paper efficiently probe the frontier defined by cost versus impact, in order to rapidly de-risk an optimal solution $x^*$ to \eqref{eq:generic} by either
\begin{itemize}
    \item [(a)] computing an alternative 
vector $\hat x \in P$ with \textit{significantly lower impact but moderate increase in cost} relative to $x^*$, or
    \item [(b)] proving that no such vector exists.
\end{itemize}
  Thus, in simple terms, we develop computational tools that support an \textit{optimistic} risk-tolerance stance, when possible, and that yield at least a partial explanation (i.e., an impossibility proof) when not achievable.

\subsection{Motivation}\label{subsec:motivation}
There is a large literature pointing to undesirability of concentration arising from optimization in varied settings such as finance and logistics. See, e.g., \cite{randconcentration}\footnote{In particular see page 13, where high dimensionality and opacity of risk are brought up.} %"For example, having
%diversified supply sources for a particular part of the supply chain might not mean much if the
%disruption is from somewhere unexpected. Companies do not know what they do not know."}. 
Here we present an example to fix ideas.

The plot in Figure \ref{fig:log1} was produced by solving a logistical MIP that
models shipping a variety of commodities using several vehicle types on a 
medium-sized network in a multiperiod setting.  The model is endowed with 
several types of capacities, vehicle availabilities, deadlines and vehichle-commodity compatibility rules. The particular case considered here has over 400,000 variables of which over 5,000 are integral.  

Figure 1 displays sorted
\textit{activity} weights -- an activity is the set of all shipments
on a given network link at a given point in time, over all available vehicles. Even though there are over 10,000 such activities, only approximately 150 are nonzero at the optimum.  More significantly, we see very high concentration near the top. 
\begin{figure}[H]
  \centering
  \resizebox{\textwidth}{!}{
  \begin{minipage}[b]{0.49\textwidth}
    \includegraphics[scale=0.48]{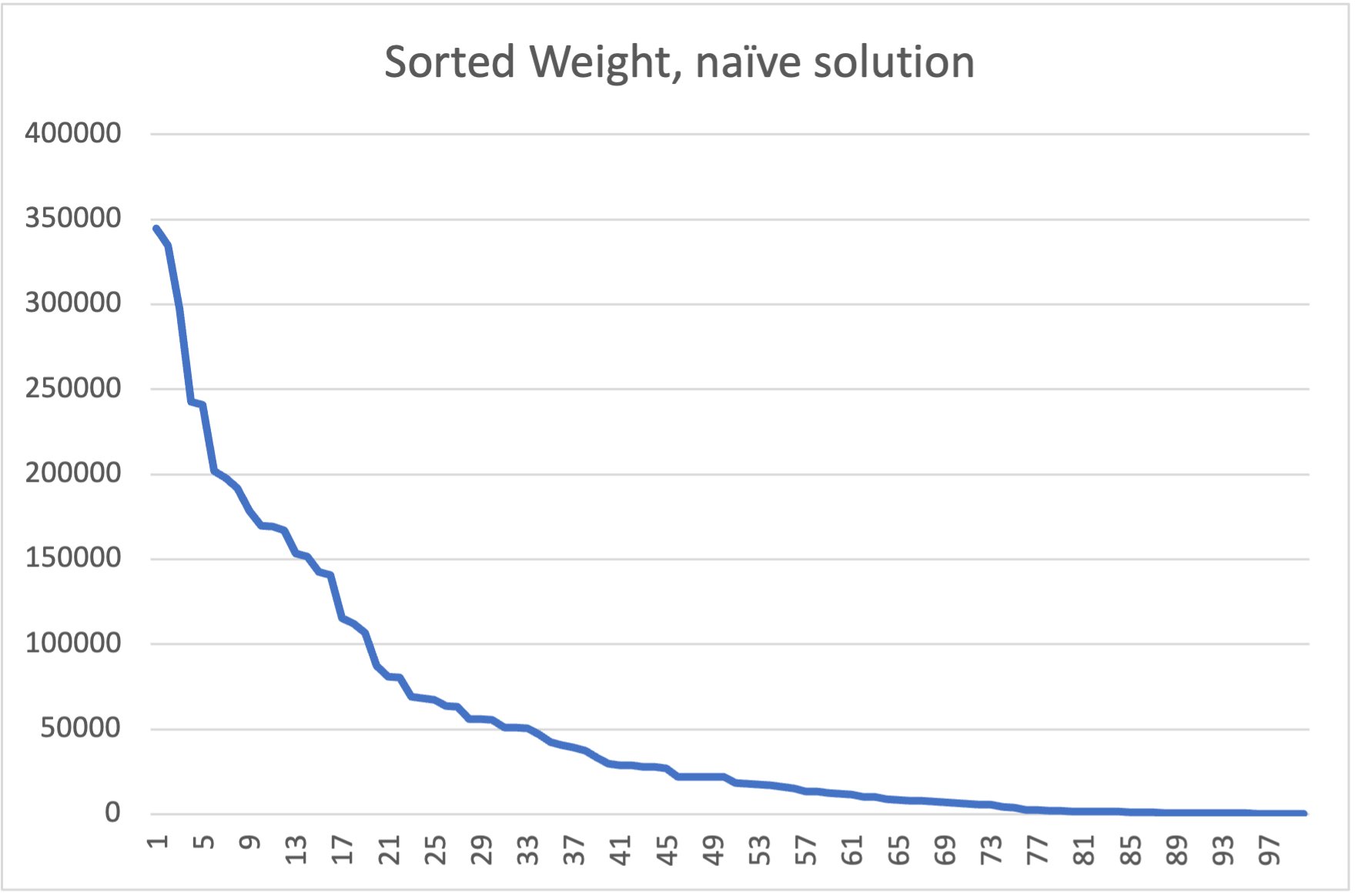}
    \caption{Sorted activity weights.}\label{fig:log1}
  \end{minipage}
  \hfill
  \begin{minipage}[b]{0.49\textwidth}
    \includegraphics[scale=0.48]{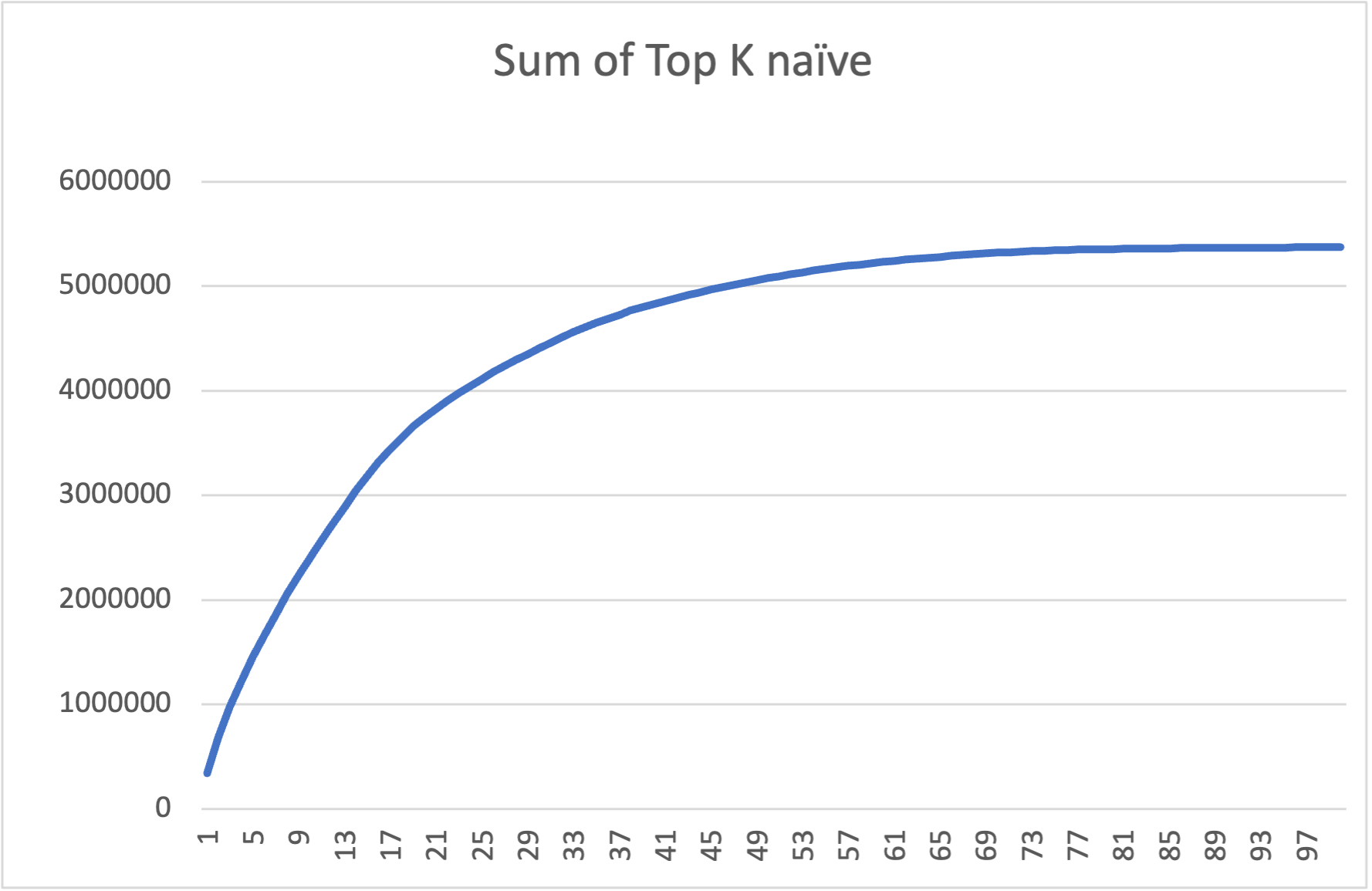}
    \caption{Aggregate weight of top activities.}\label{fig:log2}
  \end{minipage}
  }
\end{figure}   
Figure \ref{fig:log2}  plots the aggregate weight of the top $K$ activities, for $K = 1, 2, \ldots$ 
Note that, for example, the top ten activities   account for approximately $40 \%$ of all shipped weight.
Such concentration is risk-inducing in an intuitive fashion: congestion or an adversarial event that affects a small number of  activities will have a negative 
impact out of proportion to the value of the concentration.

As a simple example of the computational challenges inherent in addressing concentration risk, consider a logistical setting where there are decision variables of the form $x_{j,t}$, modeling aggregate commodity levels present at location $j$ at time $t$. Suppose we define
\begin{align} \label{eq:minexample}
\Phi(x) \, \doteq & \, \max_{H \in \cI} \sum_{t \in H} \sum_{j \in \cS_H} x_{j,t}
\end{align}
where $\cI$ is a family of short time intervals  and for each interval $H$, $\cS_H$ is a set of locations, possibly dependent on $x$ itself. An example is provided by the "top $K$" setup ($K = 2, 3, \ldots$, small) where for $H \in \cI$, $\cS_H$ is the set of $K$ highest values $x_{j,t}$ among all locations $j$ and all time periods $t$ within $H$. In that case, \eqref{eq:minexample} takes the form
\begin{align} \label{eq:minordered}
\Phi(x) \, \doteq & \, \max_{H \in \cI} \sum_{t \in H} \sum_{k = 1}^K  x_{(k),t}
\end{align}
where for $1 \le k \le K$, $x_{(k),t}$ is the $k^{th}$ largest value $x_{j,t}$ with $t \in H$.  Thus, $\Phi(x)$ is large when there is high concentration of commodities in a narrow time frame at a small set of locations. A computational challenge arises in that an explicit representation of $\Phi(x)$ will require a high dimensional formulation.  Additionally, rather than using an aggregate quantity $x_{j,t}$ it may prove useful to refer to specific commodities, further increasing dimensionality. Incorporating a model of exogenous uncertainty into $\Phi(x)$ will further extend the dimensionality.

For a very different example consider an optimization problem
\begin{align} \label{eq:nonlineargeneric}
\min & \ f_0(x) \quad \text{s.t.} \   \sum_{j \in J_i} a_{ij} f_{ij}(x) \ \leq \ b_i, \ \forall i \in I, \ x \in \bR^n, 
\end{align}
%\begin{subequations} \label{eq:nonlineargeneric}
%\begin{align}
 % \min & \ f_0(x) \nonumber \\
  %\text{s.t.} & \ \sum_{j \in J_i} a_{ij} f_{ij}(x) \ = \ b_i, \quad i \in I \nonumber \\
 % & \ x \in \bR^n, \nonumber
%\end{align}
%\end{subequations}
\noindent where $a_{ij} \neq 0$ and $f_{ij}(x)$ is a given function (e.g., a monomial) for
each $i \in I$ and $j \in J_i$.  Suppose the $a_{ij}$ are uncertain, under a model where $a_{ij}$ is replaced by $a_{ij}(1 + z_{ij})$ with $z$ is chosen from within some set $\cZ$.  We can thus define 
\begin{align} \label{eq:maxinfeas}
\Phi(x) \, \doteq & \, \max_{z \in \cZ} \max_{i\in I} \left\|\left[ \sum_{j \in J_i} a_{ij}f_{ij}(x) (1+z_{ij}) - b_i \right]^+ \right\|_p
\end{align}
for some $p \ge 0$ as a measure of the adversarial infeasibility incurred by
$x$. 
%Note that for a given $x$, in general $\Phi(x) > 0$ if $\cZ$ allows for some $z_{ij} f_{ij}(x) \neq 0$ because \eqref{eq:nonlineargeneric} is equality constrained.  %An explicit description of $\Phi(x)$ suitable for optimization is likely to be both high-dimensional and complex especially when the $f_{ij}(x)$ are nonlinear: for example, we
%expect that computing $\min_{x \in \bR^n} \Phi(x)$ will %in general be a difficult problem.      
Our algorithms seek to compute $x$ feasible and near-optimal for \eqref{eq:nonlineargeneric} that achieves (much) smaller $\Phi(x)$, i.e., much smaller maximum infeasibility error than a nominally optimal solution to \eqref{eq:nonlineargeneric}.
\subsubsection{Discussion of a disadvantageous approach}
Our de-risking goal can, in principle, be addressed by solving impact-constrained formulations   
\begin{align} \label{eq:generic-adversarial}
x_{\Lambda} \ \doteq \ \argmin  \{ \, c(x) \, : \, \, x \in P, \ \ \Phi(x) \, \le \, \Lambda \, \}, 
\end{align} 
where $\Lambda$ is an (explicit) input value.  We next argue that this is an unfavorable methodology. \\

\noindent {\bf 1. Nontrivial selection for $\Lambda$.}  Problem \eqref{eq:generic} will typically be nonlinear, nonconvex, or both, and  $\Phi(x)$ will be nonlinear.
%A priori, it may  not be clear what value $\Lambda$ should take, and why. %Any choice for $\Lambda$ reflects a risk-tolerance stance, and, ideally, one should estimate an efficient frontier by solving \eqref{eq:generic-adversarial} for multiple values $\Lambda$. Additionally, t
Hence  $ x_\Lambda$   and   $c(x_\Lambda)$ may both be sharply sensitive to  $\Lambda$, an undesirable behavior.    An additional hazard arises from a computational standpoint: a choice that is aggressive --$\Lambda$ is too small-- would cause problem \eqref{eq:generic-adversarial} to become infeasible or nearly so \footnote{In fact the computation of
the smallest value $\Lambda$ such that \eqref{eq:generic-adversarial} remains feasible can be a nontrivial problem even in cases where \eqref{eq:generic} is a linear: it includes the min-cut problem as a special case.}.  In general, especially in the nonlinear and nonconvex case, a specific choice for $\Lambda$ can give rise to solution times that are unacceptably lengthy.  In fact the solver might simply fail, while, potentially, a slightly different choice yields a much faster solution.  \\

\noindent {\bf 2. An overly complex and large formulation.} As highlighted in the examples underlying equations \eqref{eq:minordered} and \eqref{eq:nonlineargeneric}, the explicit representation of $\Phi(x)$ in \eqref{eq:generic-adversarial} might require a formulation of  unacceptably large size, causing problem
\eqref{eq:generic-adversarial} to become far more challenging than \eqref{eq:generic}. In both cases 
$\Phi(x)$ is of the form $\max_{i \in I} \phi_i(x)$ where $I$ is a very large set and the $\phi_i(x)$ functions are nonlinear and require a description which is, itself, nontrivial, even in a purely deterministic setting.  

Another nominally equivalent approach resorts to solving
\begin{align}
& \min \Phi(x), \ \text{s.t.} \ x \in P, \ \ c(x) \, \le \, \Gamma,
\end{align}
for an input $\Gamma$.  The challenge here is that formal, constrained minimization of the (in general) very complex function $\Phi(x)$ will likely be difficult.  The approach is possibly also excessively conservative and inflexible with regards to the choice of $\Gamma$\footnote{That is to say, a slightly different choice for $\Gamma$ might yield faster solution times and nearly equal risk.}. 
\subsubsection{Our approach}
We develop cutting-plane algorithms that aim to parsimoniously de-risk an optimal solution $x^*$ to the nominal problem \eqref{eq:generic}, i.e., to compute $\hat x \in P$ with  small increase in cost together with large decrease in impact, both relative to $x^*$, or to prove that this goal is not possible.  While such  de-risking  cannot be guaranteed (because of structure of the feasible set $P$), experiments indicate that it
    is frequently achievable in relevant cases. 

Our procedures  are motivated by the impact-weighted formulation 
\begin{align}\label{eq:generic-weighted}
\min & \ \ c(x) \ + \Theta \, \Phi(x), \quad \text{s.t.} \ \  x \in P,
\end{align} 
where $\Theta > 0$ is a risk-aversion parameter\footnote{The classical application of this idea is found in \cite{markowitz}.}. Of note, in a typical run the procedures will terminate \textit{without}  solving \eqref{eq:generic-weighted} to proved optimality. 

In particular, the cutting-plane Algorithm \ref{alg:SOFTMAX-adv}, Theorem \ref{thm:finiteconv} and Lemma \ref{le:formaltheta} yield the following.  Let $x^* \in P$, $0 < \lambda^{\text{lo}} < \lambda^{\text{hi}} \le 1$ and $0 < \xi$.  Under fairly general conditions, in a finite number of iterations of Algorithm \ref{alg:SOFTMAX-adv}, we either 
\begin{itemize}
    \item[(i)] obtain $\hat x \in P$ with $\Phi(\hat x) \le \lambda^{\mathrm{hi}} \Phi(x^*) \ \text{and} \ c(\hat x) \, \le \, c(x^*)\left(1 + \frac{\lambda^{\mathrm{hi}}}{\lambda^{\mathrm{hi}} - \lambda^{\mathrm{lo}}} \xi \right)$, or 
    \item [(ii)] prove (within tolerances) that there is \textit{no} $x \in P \ \text{with} \ \Phi(x) \le \lambda^{\mathrm{lo}} \Phi(x^*) \ \text{and} \ c(x) \le (1 + \xi) c(x^*).$ 
\end{itemize}
The value $\Theta$ that yields (i) and (ii) is chosen as a function of all the inputs listed above. Note that to a partial extent, any choice for $\Theta$ in \eqref{eq:generic-weighted} implies an equivalent 
choice for $\Lambda$ in \eqref{eq:generic-adversarial} -- however, this equivalence only applies to formally optimal solutions to the respective problems, a feature bypassed in our approach.

In the above descriptions, the impact function $\Phi(x)$ is presented in abstract form. Potentially, this includes an adversarial ingredient.  As an example in the logistical setting given above,
consider a family of joint distributions $\cZ$, each of which describes an 
uncertain event that results in capacity loss in network links. The risk-aligned quantity
of interest, here, would be a measure (e.g., a value-at-risk) of the amount of  flow that is 
delayed or lost under distribution $z \in \cZ$.  

Formally, the problem setup we will consider has the following characteristics. As inputs, we are  given a compact set $\cZ$ that parameterizes uncertainty, and for each $z \in \cZ$,   
\begin{itemize}
    \item [(a)] A set $I$ of nonnegative-valued functions of the form $\phi_i(x|z)$ ($i \in I$).   We assume that $I$
is finite though potentially very large.  We will refer to the $\phi_i$ as the
\textit{features}.  Each feature is a function of the decision vector $x$ as well
as the exogenous risk-inducing vector $z \in \cZ$. 
\item [(b)] A set of values
$\lambda_1 \ge \ldots \ge \lambda_{|I|} \ge 0$ with $\sum_{i = 1}^{|I|} \lambda_i = 1$. 
For $x \in P$ we define  the
impact on $x$ of event $z$
$$ \Phi(x|z) \ \doteq \ \sum_{k = 1}^{|I|} \lambda_{k} \phi_{(k)}(x \, | \, z), $$
where $\phi_{(k)}(x \, | \, z)$ is the $k^{th}$ largest value $\phi_i(x \, | \, z)$.  
We set
$$ \Phi(x) \, \doteq \, \max_{z \in \cZ} \Phi(x|z).$$
Thus 
problem \eqref{eq:generic-weighted}, in its epigraph representation, takes the form
\begin{subequations} \label{eq:ordfeatures0-epi}
\begin{align}
\min & \ c(x) \ + \ \Theta \, \Phi_L\\
\text{s.t.} \ & \ x \in P, \quad \  \Phi_L \, \ge \, \max_{z \in \cZ} \, \sum_{k = 1}^{|I|} \lambda_{k} \phi_{(k)}(x \, | \, z). \label{eq:ordfeatures1_Lepi}
\end{align}
\end{subequations}  
\end{itemize}

\noindent Denote by $\cS_{|I|}$ the set of permutations of $|I|$ entities, and for $\pi \in \cS_{|I|}$  (with a slight abuse of notation) use $\pi(i)$ to denote the corresponding element $i$ of $I$. Thus, for $x \in P$ and $z \in \cZ$, we have 
\begin{align}\label{eq:factorial}
&  \sum_{k = 1}^{|I|} \lambda_{k} \phi_{(k)}(x \, | \, z)\ = \ \max_{\pi \in \cS_{|I|}} \ \sum_{k = 1}^{|I|} \lambda_{k} \phi_{\pi(k)}(x \, | \, z).
\end{align}

\noindent In particular, when
$\lambda_k = 0$ for $k > 1$ problem \eqref{eq:ordfeatures0-epi} becomes
\begin{subequations} \label{eq:maxfeatures-epi}
\begin{align}  
\min & \ c(x) \ + \Theta \, \Phi_L\\
\text{s.t.} \ & \ x \in P, \quad \  \Phi_L \, \ge \, \max_{z \in \cZ} \max_{i \in I} \phi_i(x \, | \, z). \label{eq:maxfeatures_Lepi}
\end{align}
\end{subequations}  

\noindent Using \eqref{eq:factorial}, we note that \eqref{eq:ordfeatures0-epi} is, in fact, a \textit{special case} of
\eqref{eq:maxfeatures-epi} with appropriately redefined features (e.g., one for each $\pi \in \cS_{I}$). \textit{For simplicity of language}, in the theoretical results below we assume the form \eqref{eq:maxfeatures-epi}. Our computational experiments include cases with, e.g., $K = 10$.

This paper is organized as follows.  Section \ref{sec:review} describes related work and Section \ref{sec:algo} presents our methodological contributions which are
analyzed in Section \ref{subsec:analysis} (in particular, see Theorem \ref{thm:finiteconv}).  Section \ref{subsec:introex} presents simple applications of our algorithms, with more substantial experiments given 
in Section \ref{sec:numericalexper}.

\subsection{Review of prior work}\label{sec:review}
The perspective we take in this paper concerns risk arising from solutions to 
optimization problems that incorporate concentration, in a generic sense.  One of the earliest citations in this context is \cite{fratta1973flow}, which concerned routing problems (in packet networks) that can formally be stated as linear optimization problems.  However, congestion brings about the risk of queueing delays: let the flow on an arc with capacity $u$ be $x$: then the M/M/1 queueing delay is $u/(u - x)$ (see, e.g., \cite{kleinrock75}).  The combinatorial algorithm in \cite{fratta1973flow} approximately solves an optimization problem whose objective function is the original cost function, plus the sum of all queueing delays, thus trading-off cost versus queueing risk (i.e., delays).

A related problem was later taken up in \cite{ShahrokhiMatula}.  They consider a capacity-constrained multicommodity flow routing problem, where the capacities are insufficient to route all commodities.  Under such conditions, the algorithm in
\cite{ShahrokhiMatula} computes a routing that approximately  minimizes the maximum overload in any arc, which is defined as the ratio of flow to capacity.  The algorithm replaces this min-max task, with an approximately equivalent nonlinear optimization problem which is (again, approximately) handled using a first-order method.  

This approach was greatly extended and improved in \cite{PlotkinST91}, and \cite{grigoriadis1994fast}. Let $\cA$ denote the set of arcs in a multicommodity flow network, and for an arc $(i,j) \in \cA$, let $u_{ij}$ denote its capacity, and
for a (multicommodity) flow vector $f$ let $f_{ij}$ denote the sum of commodity flows routed on $(i,j)$.  Then the min-max routing problem described in the above 
paragraph can be stated as 
\begin{align}\label{eq:minmaxcong}
 & \min_{f \in \cF} \max_{(i,j) \in \cA} f_{ij}/u_{ij},
\end{align}
where $\cF$ is the set of feasible flows, i.e., all flows that deliver all commodities from their source to their destination, possibly exceeding capacities.  On very large networks with 
many commodities this optimization problem, which can be stated as a linear program, proves quite challenging to modern LP solvers.  In contrast, \cite{PlotkinST91} and \cite{grigoriadis1994fast} show that the min-max problem
can approximately be modeled using an appropriate nonlinear \textit{potential function} as a 
proxy:
\begin{align}\label{eq:potential}
& \min_{f \in \cF} \ \ln \left(\sum_{(i,j) \in \cA} e^{\alpha f_{ij}/u_{ij}}\right).
\end{align}
More precisely, given $0 < \epsilon < 1$, choosing $\alpha$ appropriately (primarily, $\alpha$ proportional to $\epsilon^{-1}$) a solution to the potential function problem \eqref{eq:potential} yields a flow vector $f$ that is $\epsilon$-optimal to the min-max problem \eqref{eq:minmaxcong}.  Moreover, the approximate 
minimization in \eqref{eq:potential} is carried out, effectively, via a first-order method where 
individual iterations amount to linear programs whose objective is the gradient
of the function in \eqref{eq:potential}.  We will rely on
a similar proxy in our work, though our algorithmic goals
and methodology are quite different.

As a function of $\epsilon$, the number 
of iterations needed to obtain an $\epsilon$-optimal solution to \eqref{eq:minmaxcong} grows proportional to $1/\epsilon^2$ (later improved to $1/\epsilon$, 
\cite{BienstockIyengar2004}).  See \cite{bienstock2006potential} for an analysis
of \cite{fratta1973flow}, \cite{PlotkinST91} and \cite{grigoriadis1994fast} and follow-up work that improved on these algorithms, as 
well as computational experiments.

Our risk perspective can include an exogenous adversarial component.  There is a very large literature of models for \textit{interdiction} that take an adversarial standpoint: how to optimally or maximally
disrupt, e.g., a logistical operation.  In this regard, the case of a 
network under stochastic capacity interdiction is studied in  
\cite{Cormican1998StochasticNI} (also see references therein).  This
work considers a variety of models where an adversary interdicts 
arcs of a network by reducing their capacity under various assumptions of likelihood of success of interdiction and of capacity reduction, as well as adversarial capabilities; and produces efficient 
formulations for such problems.
Optimization problems where under the risk of adversarial interdiction are often handled using a cutting-plane algorithm that is reminiscent of (or equivalent to) Benders' decomposition 
\cite{benders1962}.    Related problems are considered in 
\cite{brown2006defending}.

A salient and challenging  application of interdiction analysis involves the so-called $N-K$ problem in power grids; see \cite{swb2}, \cite{bienstock2010n}.  Briefly, an operational constraint in modern power grids is that \textit{any} simultaneous disablement of
up to $K$ components (e.g., arcs) should be tolerated by the system; here $K > 0$ is a small integer. The case $K = 1$ is enshrined as law in many countries, but larger values of $K$ 
are becoming relevant. Note that the enumeration of all subsets
of size $K$ is infeasible even for relatively small $K$ (e.g., 3) when the grid is large, and a generic Benders' approach, when applicable, proves
invaluable. 

Optimization problems under interdiction risk are usually (if unwittingly) modeled as bilevel optimization problems, for which there is also a very abundant literature.  Bilevel optimization problems are complex and Benders' decomposition is not always an easy or even feasible task.  For a recent survey, see \cite{schmidt2026}.

Finally, the incorporation of an adversarial model can bring about the robust optimization perspective,
especially in the so-called "risk-budgets" setup.  See \cite{bertsimas2004price} and \cite{bertsimas2011theory}. Many of the previously cited works on interdiction models can be seen through the lens of risk budgets.

\section{Algorithms}\label{sec:algo}
We describe a procedure, Algorithm \ref{alg:SOFTMAX-adv},  which has provable convergence attributes.  Notably, this procedure does
not necessarily solve problem \eqref{eq:maxfeatures-epi} to optimality, however it provably produces an output vector with desirable attributes.

This algorithm serves as a 
common intellectual core for a family of procedures that prove empirically successful. These procedures, which are discussed in Section \ref{subsec:modifications}, incorporate  pragmatic modifications to Algorithm \ref{alg:SOFTMAX-adv} for high-dimensional cases and in some cases are also 
endowed with provably correct attributes.
 
\newpage
%for any $x \in P$ and $z \in \cZ$ we write $$\phi_{\max}(x|z) = \max_{i \in I} \phi_i(x|z)$$ which is attained since $I$ is finite.  
\begin{algorithm}
\caption{SOFTMAX-ADVERSARIAL }\label{alg:SOFTMAX-adv}
\begin{algorithmic}[1] 
\State {\bf Inputs:}  parameters $\Theta > 0$; $\alpha > 0$;  positive tolerances $\ \Delta, \delta < 1, \, \Delta'$; max. iteration count $\,t^{\max}$.
%\State {\bf Inputs:} $t^{\max} > 0$ (max. no. of iterations), $\Theta > 0$,  $\Delta, \, \Delta' > 0$ (tolerances),  $\phi_{\text{target}} \ge 0$ and  $\alpha > 0$.
\State Set $t = 0$, and initialize the \emph{master problem} as
$$ \min \ c(x) + \Theta \, \Phi_L, \quad \text{s.t.} \ x \in P, \ \ \Phi_L \ge 0.$$ 
\While{$t < t^{\max}$}:
\State Solve the master problem; let $(x^{t},\Phi_L^{\ t})$ be an optimal solution.  
\State {\bf Boosting step.} Compute $z^{t} \in \cZ$ such that $$ \ln\sum_{i \in I} e^{\alpha \phi_i(x^{t}|z^t)}     \ \ \ge \ \ \max_{z \in \cZ} \ln \sum_{i \in I}     e^{\alpha \phi_i(x^{t}|z)}      \quad - \quad \Delta'.$$
\quad \ \ Define $\displaystyle \phi_{\max}^t = \max_{i \in I} \phi_{i}(x^{t}|z^t).$
\If{(a) $ \phi_{\max}^t \, \le \, \Phi_L^{\,t} + \Delta$, or (b)  $ \phi_{\max}^t-\Phi^{\, t}_L \le \delta \, \phi_{\max}^t$ }  {\bf STOP}.
%\If{$ \phi_{\max}(x^{t}|z^t) \, \le \, \Phi_L^{\,t} + \Delta$ {\bf or} $\frac{\phi_{\max}(x^{t}|z^t) - \Phi_L^{\,t}}{\phi_{\max}(x^{t}|z^t)} \le \Delta$, }  {\bf STOP}.
%\If{(a) $   
%\phi_{\max}(x^{t}|z^t) \, \le \, \Phi_L^{\,t} + \Delta$, or (b) $     \phi_{\max}(x^{t}|z^t) \, \le \, \phi_{\text{target}}, $} {\bf STOP. Currently NOT using (b)}
%\If{$   \ln \sum_{i \in I}     e^{\alpha \phi_i(x^{t}|z^t)} \, \le \, \Phi_L^{\,t} + \Delta$} {\bf STOP.}
\Else 
\State  For each $i \in I$ write $\pi^{t}_i = \frac{e^{\alpha \phi_i(x^{t}|z^t)}}{\sum_{j \in I} e^{\alpha  \phi_j(x^{t}|z^t)}  } $ 
\State {\bf Separation.} Add the cut 
$  \,  \Phi_L \ \ge \  \sum_{i \in I} \pi^{t}_i \, \phi_i(x|z^t)$ to the master problem.
\EndIf 
\State $t \gets t+1$.
\EndWhile
%\EndProcedure
\end{algorithmic}
\end{algorithm}

\subsection{Preliminary discussion of Algorithm \ref{alg:SOFTMAX-adv}}\label{subsec:discussion}
For $x \in X$ recall that  $\Phi(x) =\max_{z \in \cZ} \max_{i \in I} \phi_{\text{i}}(x|z)$. 
%i.e., the right-hand side of constraint \eqref{eq:maxfeatures_Lepi}. Consider any iteration  $t' \ge 1$. Note that $\pi^{t'} \ge 0$ and $\sum_i \pi^{t'}_i = 1$.  As we will
Consider any iteration $t \ge 0$. Note that $\pi^{t} \ge 0$ and $\sum_i \pi^{t}_i = 1$.  As we will
see below (Lemma \ref{le:rhoLbelos}), this implies $\Phi^{\, t}_L \le \Phi(x^t)$, and,
as a result, the master problem is always a relaxation of \eqref{eq:maxfeatures-epi}.
Moreover, 
\begin{itemize}
\item $\Phi(x^t) - \Phi^{\, t}_L$  amounts to the infeasibility of $(x^t, \Phi^{\, t}_L)$ in 
problem \eqref{eq:maxfeatures-epi} -- should this quantity be small, the algorithm will have 
converged to an approximate optimal solution.  
    \item By definition $\displaystyle \phi^t_{\max} \doteq \max_{i\in I} \phi_i(x^t | z^t) \, \le \Phi(x^t)$.
    \item Furthermore  (Lemma \ref{le:phimaxok}), $\Phi(x^t) \le \phi^t_{\max} + 2\Delta$.  
\end{itemize}
In summary, $\phi^t_{\max}$ is a valid proxy to $\Phi(x^t)$, and, as a result, the 
stopping conditions in line 6 are numerically valid termination conditions for the cutting plane
algorithm (absolute and relative, resp.). In order for this statement to be mathematically valid
we need that $\delta$ be small in absolute terms, e.g., $\delta = 10^{-6}$, and that $\Delta$ be small
be small in relative terms, i.e. $\Delta \le 10^{-6} \phi_{\max}^t$.  In the numerical implemenation of
Algorithm \eqref{alg:SOFTMAX-adv} we will relax these conditions in order to allow early termination, in line with our 'de-risking' goal.  \\

\noindent {\bf Remark.} In machine learning language, the function optimized in line 5 of the algorithm
is known as  \textit{log-sum-exp}, while the cut coefficients in line 8 follow the \textit{softmax} function.  Formally, given $y \in \bR^n$, and $\alpha > 0$ we write
 $\text{SOFTMAX}(\alpha,y)_i \ \doteq \ \frac{e^{\alpha y_i}}{\sum_{j = 1}^n e^{\alpha y_j}}.$
It is seen that the softmax function is proportional to the gradient of log-sum-exp (in feature space), that is to say,
defining $\phi = (\phi_i, \, i \in I)$, $L(\phi) =  \ln\sum_{i \in I} e^{\alpha \phi_i} $ and $g(\phi)$ as the vector with entry $g_i(\phi) = \alpha \frac{e^{\alpha \phi_i}}{\sum_{j \in I} e^{\alpha  \phi_j}}  $ for each $i \in I$, we have $ \nabla_{\phi} L(\phi)   \ = \ g(\phi)$.

As defined, $L$ is a convex function of $\phi$. Thus, given a particular vector $\bar \phi$ it
holds that
\begin{align}
& L(\phi) \ \ge \ g^T(\bar \phi) (\phi - \bar \phi) \, + \, L(\bar\phi). \label{eq:kellylogexp}
\end{align}
Using this observation, one can write a modification to Algorithm \ref{alg:SOFTMAX-adv} similar to Kelley's \cite{kelley60} algorithm
using the outer-approximation cuts \eqref{eq:kellylogexp} instead of the cuts added in line 9.  

\subsection{Analysis of Algorithm \ref{alg:SOFTMAX-adv}}\label{subsec:analysis}
We next present results regarding the convergence properties of Algorithm \ref{alg:SOFTMAX-adv}, concluding with Theorem \ref{thm:finiteconv} and Corollary \ref{co:reallyfinite}. 
%In any iteration $t$, write $\phi^t_U \doteq \Phi(x^t)$, i.e., $\max_{z \in \cZ} \max_{i \in I} \phi_{\text{i}}(x^t|z)$ according to our notation above.  
We will set $\Delta' = \alpha \Delta$ with $\alpha \ge \frac{\ln |I| + [ \ln( \,2\Phi(x^0)/\Delta \,)]^+}{\Delta} $, which is well-defined since $\alpha$ is not actually needed until the first execution of line 5, i.e., not until after
 $x^0$ is computed. 
\begin{LE} \label{le:rhoLbelos}At any iteration $t$, $\Phi^{\,t}_L \ \le \ \Phi(x^t)$.
\end{LE}
\begin{PROOF} This is true for $t = 0$ since $\phi^0_L = 0$, and, for $t > 0$, since $\Theta > 0$ we have, for some $t' < t$, that  
$ \phi^{ \, t}_L \ = \ \sum_{i \in I} \pi^{t'}_i \, \phi_i(x^t|z^{t'}) \ \le \ \max_{i \in I} \phi_i(x^t| z^{t'})$,  since the $\pi^{t'}_i$ are nonnegative and sum to $1$. \end{PROOF}
\begin{LE} \label{le:phimaxok} At any iteration $t$, $ \phi_{\max}^t \, \le \, \Phi(x^t) \, \le \, \phi_{\max}^t \, + \, 2\Delta$.\end{LE}
\begin{PROOF}
The first inequality follows by definition of $\phi_{\max}^t$.  Next, given $z \in \cZ$, $e^{\alpha \max_{i \in I} \phi_{i}(x^t|z)} \, \le \, \, \sum_{i \in I} e^{\alpha\phi_i(x^t | z)}\le \, |I| e^{\alpha \max_{i \in I} \phi_{i}(x^t|z)}$ and therefore
%\begin{subequations}
\begin{align}\label{eq:key1}
 & \alpha \, \max_{i \in I} \phi_{i}(x^t|z) \quad \le \quad \ln \sum_{i \in I} e^{\alpha \phi_i(x^t | z)} \quad \le \quad \ln |I| + \alpha  \max_{i \in I} \phi_{i}(x^t|z).  
\end{align}
% \end{subequations}
So $$\alpha \Phi(x^t)  \ \le \ \max_{z \in \cZ} \ln \sum_{i \in I} e^{\alpha \phi_i(x^t|z)} \ \le \ \ln \sum_{i \in I} e^{\alpha \phi_i(x^t|z^t)} + \Delta' \ ,$$
as per line 5 of the algorithm. Using \eqref{eq:key1} with $z = z^t$, the rightmost quantity is at most
$\ln |I| + \Delta' + \alpha \phi_{\max}^t. $
%\phi_{\max}(x^t | z)$, %and, also, $\alpha \phi_{\max}(x^t|z^t)} \, \le \, \ln \sum_{i \in I} e^{\phi_i(x^t | z^t)}$.  
In summary,
$ \Phi(x^t) \ \le \frac{\ln |I| + \Delta'}{\alpha} \ +  \ \phi_{\max}^t \, \le \, \phi_{\max}^t  +  2\Delta$ by definition of $\alpha$ and $\Delta'$.
\end{PROOF}
\begin{CO} \label{co:inbetween} Suppose the algorithm terminates at iteration $t$ (line 6). If criterion (a) applies,  $\Phi(x^t) - 3\Delta \, \le \, \Phi^{\, t}_L \, \le \,  \Phi(x^t)$. And if (b) applies
$(1 - \delta)\Phi(x^t) - 2\Delta \, \le \, \Phi^{\, t}_L  \, \le \,    \Phi(x^t)$. \end{CO} 
\begin{PROOF} This follows from the termination condition, Lemmas \ref{le:rhoLbelos} and \ref{le:phimaxok}.  

\end{PROOF}
\begin{OB} \label{ob:flip}Given $t$, let $\tilde x \in P$ be such that $c(\tilde x) < c(x^t)$. Then $\Phi(\tilde x) > \Phi^{\, t}_L$.  
\end{OB}
\noindent This follows because $(\tilde x, \Phi(\tilde x))$ is feasible for the master problem at iteration $t$.  In light of Corollary \ref{co:inbetween}, this shows that even prior to
termination, the algorithm yields useful information.

\begin{LE} \label{le:varrhoupper} For any $t$, $\Phi^{\, t}_L \le \Phi(x^0)$.
\end{LE}
\begin{PROOF}
For any $x \in P$, the pair $(x,\Phi(x))) $ is feasible for the master problem at iteration $t$. Using this fact with $x = x^0$, we get 
$ c(x^t) + \Theta \, \Phi^{\, t}_L \le c(x^0) + \Theta \,\Phi(x^0).$
Since by construction $c(x^0) \le c(x^t)$, and $\Theta > 0$, the proof is complete.
\end{PROOF}

In the proofs that follow, for a given $t$ we define
$$S^t = \{i \, : \, \phi_i(x^t|z^t) \le \Phi^{\, t}_L\} \quad \text{and} \quad B^t = \{i : \, \phi_i(x^t|z^t) > \Phi^{\, t}_L\}.$$
\begin{LE} \label{le:halfviol}
   Suppose that at iteration $t$ the algorithm does not stop on line 6.  Then the cut on line 9 is violated by $(x^t, \Phi^{\, t}_L)$ by more than $\frac{\sum_{i \in B^t} e^{\alpha \phi_{i}(x^t|z^t)}}{\sum_{i \in I} e^{\alpha  \phi_i(x^{t}|z^t)}} \Delta/2$.
\end{LE}
\begin{PROOF}
It suffices to show that 
\begin{align}
& \sum_{i \in I} e^{\alpha \phi_i(x^t|z^t)}\Phi^{\, t}_L \ + \ \sum_{i \in B^t} e^{\alpha \phi_{i}(x^t|z^t)} \, \Delta/2 \ < \ \sum_{i \in B^t} e^{\alpha \phi_i(x^t|z^t)} \phi_i(x^t|z^t), \quad \text{i.e., that} \\
& \sum_{i \in S^t} e^{\alpha \phi_i(x^t|z^t)}\Phi^{\, t}_L \ < \ \sum_{i \in B^t} e^{\alpha \phi_i(x^t|z^t)} (\phi_i(x^t|z^t) - \Phi^{\,t}_L - \Delta/2).
\end{align}
% comment: divide the first row by $\sum_{i \in I} e^{\alpha \phi_i(x^t|z^t)}$
%\quad \text{which is implied by}\\
%& |S_t| e^{\alpha \Phi^{\, t}_L} \Phi_L^T \ < \ 
%\end{align}
The left-hand side is upper bounded by $|S_t| e^{\alpha \Phi^{\, t}_L} \Phi^{\, t}_L < |S_t| e^{\alpha (\phi_{\max}^t - \Delta)} \Phi^{\, t}_L$, while the right-hand side is lower bounded by $e^{\alpha \phi_{\max}^t} (\phi_{\max}^t - \Phi^{\, t}_L - \Delta/2) \ > \ e^{\alpha \phi_{\max}^t} \Delta/2$.  Hence the result follows
if we can argue that
$$ |S_t|  \Phi^{\, t}_L \ < \ e^{\alpha \Delta} \Delta/2.$$
Using $|S| \le |I|$ and Lemma \ref{le:varrhoupper}, this holds since
$ \ln |I| + \ln \Phi(x^0) \ \le \ \alpha \Delta + \ln (\Delta/2)$,  
by choice of $\alpha$.
\end{PROOF}
\begin{CO} \label{co:finiteviol}
Under the same assumptions as Lemma \ref{le:halfviol}, the violation of the
cut added in line 9 is at least $\Delta/4$.
\end{CO}
\begin{PROOF} 
Consider the quantity $r \doteq\frac{\sum_{i \in B^t} e^{\alpha \phi_{i}(x^t|z^t)}}{\sum_{i \in I} e^{\alpha  \phi_i(x^{t}|z^t)}}$.  At least one of the terms in the numerator equals $e^{\alpha \phi_{\max}^t}$. If $|B^t| = 1$ then $r \ge \frac{e^{\alpha \phi_{\max}^t}}{e^{\alpha \phi_{\max}^t} + |S_t| e^{\alpha \Phi^{\, t}_L}} > \frac{1}{1 + |I| e^{-\alpha \Delta}}$ (because $\Phi^{\, t}_L < \phi_{\max}^t - \Delta$) and, by choice of $\alpha$, this ratio is larger than $1/2$ which concludes the proof.  Now assume $|B^t| \ge 2$, and let $p$ denote the \textit{smallest}   term in $\sum_{i \in B^t} e^{\alpha \phi_{i}(x^t|z^t)}$.  In other words, we can write $\sum_{i \in B^t} e^{\alpha \phi_{i}(x^t|z^t)} = p +  e^{\alpha \phi_{\max}^t} + q$, with $p, q \ge 0$, and so
$$ r \ \ge \ \frac{p + e^{\alpha \phi_{\max}^t} + q}{p + e^{\alpha\phi_{\max}^t} + q + |S_t| e^{\alpha \Phi^{\, t}_L}}. $$
The numerator of the derivative of this expression with respect to $p$ equals $|S_t| e^{\alpha \Phi^{\, t}_L}$ which is nonnegative, and so the expression                    is minimized when $p = 0$ and likewise with $q$. Thus, again, $r >  \frac{1}{1 + |I| e^{-\alpha \Delta}}$, as desired.
\end{PROOF}

Under appropriate conditions, Corollary \ref{co:finiteviol} implies finite termination of Algorithm \ref{alg:SOFTMAX-adv}. At any rate, this issue is addressed by the following result requires uniform continuity of the functions $\phi_i$ over $Z$. That is, for any $\epsilon>0$ there exists a $\gamma(\epsilon) > 0$ such that, for any pair $z, w \in Z$, $\|z-w\|<\gamma(\epsilon)$ implies $|\phi_i(x|z) - \phi_i(x|w)| < \epsilon$ for every $x \in P$ and any $i \in I$.

\begin{THM} \label{thm:finiteconv}
    Assume uniform continuity of the $\phi_i$ holds and let $\lambda\in(0,1)$.  Either Algorithm \ref{alg:SOFTMAX-adv} terminates finitely in line 6, or we reach an iteration $\tau$ with $\phi_{\max}^{\tau} \le \lambda \Phi(x^0)$.
\end{THM}
\begin{PROOF}
Let $n = |I|$ and define $N > \max\{n,  2/\sqrt\delta\}$. Choose 
$$\eta = \frac{\delta/2 - 2/N^2}{1 - 2/N^2}, \ \text{and} \ \alpha \ge \frac{3\log N}{\eta\Delta}.$$
\noindent Further,
\begin{itemize}
\item Define $\displaystyle \epsilon \ \doteq \ \frac{\lambda \phi^0_{U}  \delta}{4}.$
\item Let $\gamma(\epsilon)>0$ such that the uniform continuity of functions $\phi_i$ over Z holds for $\epsilon$. Define $Z_{\gamma(\epsilon)} \subseteq Z$ as a $\gamma(\epsilon)$-net for $Z$, i.e. for each $z \in Z$, there exists a $z' \in Z_{\gamma(\epsilon)}$ such that $\|z-z'\| \le \gamma(\epsilon)$. $Z_{\gamma(\epsilon)}$ (finite) is guaranteed to exist since $Z$ is compact. 
\end{itemize}
At each iteration $t$, define $\hat z^t$ and $\hat \pi^t$ as follows: 
\begin{itemize}
    \item Let $\hat z^t \in Z_{\gamma(\epsilon)}$ be the point whose $\gamma(\epsilon)$-ball covers $z^t$.
    \item For each $j$, let $\hat \pi^t_j$ be $\pi_j^t$ rounded down to the nearest integer multiple of $1/N^3$.
\end{itemize}

 For any $t$, we say that index $i$ is \textit{minor} if $\phi_i(x^t|z^t) \le (1-\eta)\phi_{\max}^t$, and \textit{major} otherwise. By our choice of $\alpha$, and since $\Delta < \phi_{\max}^t$ for every non-terminal iteration $t$, we have $e^{\alpha \phi_{\max}^t} > N^3e^{(1-\eta)\alpha\phi_{\max}^t}$. Using this bound, and denoting $\displaystyle \pi_{\max}^t = \max_{1\le i\le n}\pi_i^t$, we see that for all minor $i$, $\pi^t_i < \frac{1}{N^3}\pi^t_{\max}$. So
\begin{align*}
    \sum_{i \ minor}\pi^t_i < \frac{1}{N^2}\pi^t_{\max}\le \frac{1}{N^2}\sum_{j \ major}\pi^t_j.
\end{align*}
% $\delta < 1$ by assumption, and so $\eta < 1$. By definition of minor, there is a major index.  
Hence, since $\sum_j \pi^t_j = 1$,  
\begin{align*}
    \sum_{j \ major}\pi^t_j > 1 - \frac{1}{N^2} \ \ \text{and thus} \ \sum_{j \ major}\hat \pi^t_j > 1 - \frac{2}{N^2}, \ \text{and therefore}
\end{align*} 
% We get that $1 < (1 + 1/N^2) \sum_{j major} \pi^t_j$.  Also for $f < 1$, $1 - f < 1/(1 + f)$
\begin{align} \label{eq:bigtermbound}
    \sum_{j \ major} \hat \pi^t_j\phi_j(x^t|\hat z^t) \geq \sum_{j\ major} \hat \pi^t_j\phi_j(x^t| z^t) - \epsilon \geq (1-2/N^2)(1 - \eta)\phi_{\max}^t -\epsilon.
\end{align}
Since there are finitely many pairs $(\hat z^t, \hat \pi^t)$, the algorithm will eventually repeat a pair provided it has not already terminated. Let $\tau$ be such an iteration, meaning $\hat \pi^\tau = \hat \pi ^t$ and $\hat z^\tau = \hat z^t$ for some $t<\tau$. Hence
\begin{align*} 
    \Phi_L^\tau &\ge \sum_j \pi^t_j\phi_j(x^\tau| z^t) \ge \sum_j \hat\pi^t_j(\phi_j(x^\tau|\hat z^t) - \epsilon) \ge \sum_j \hat\pi^t_j\phi_j(x^\tau|\hat z^t) - \epsilon\\
    &= \sum_j \hat\pi^\tau_j\phi_j(x^\tau|\hat z^\tau) - \epsilon \ge (1- 2/N^2)(1-\eta)\phi_{\max}^\tau - 2\epsilon
\end{align*}
where the four inequalities follow, in order, from the cut added in iteration $t$, the definitions of $(\hat z^t, \hat \pi^t)$, the fact that the sum of the coordinates of $\hat \pi$ is at most 1, and the bound in (\ref{eq:bigtermbound}). 

Now suppose $\phi^\tau_{\max} \ge \lambda\phi_{U}^0$. Then $\displaystyle \frac{2\epsilon}{\phi^{\tau}_{\max}} \le \frac{2\epsilon}{\lambda \phi^0_{U}} \le \frac{\delta}{2}$ by our definition of $\epsilon$, and in summary,
\begin{align}
    \frac{\phi^\tau_{\max}-\Phi_L^\tau}{\phi^\tau_{\max}} \le \left[1-(1- 2/N^2)(1-\eta)\right] + \frac{2\epsilon}{\phi^\tau_{\max}} \le \frac{\delta}{2} + \frac{\delta}{2} \ = \ \delta
\end{align}
%Proof. We have that $\eta = \frac{\delta/2 - 2/N^2}{1 - 2/N^2} $. 
%So $(1 - \eta) = \frac{1 - \delta/2}{1 - 2/N^2}$, and $\left[1-(1- 2/N^2)(1-%\eta)\right] = 1 - (1 - \delta/2)$

\noindent by choice of $\eta$. Thus $(x^\tau,\Phi_L^\tau)$ satisfies the relative convergence condition (b) on line 6 and the algorithm will terminate at iteration $\tau$.

\end{PROOF}

\begin{CO} \label{co:reallyfinite} Suppose $\Phi_{\min} \doteq \min_{x \in P} \Phi(x) \ > \ 0$. Assuming
uniform continuity of the $\phi_i$, algorithm \ref{alg:SOFTMAX-adv} terminates finitely.
\end{CO}
\begin{PROOF}
Apply Theorem \ref{thm:finiteconv} using any $\lambda < \Phi_{\min}/\Phi^0_U$.
\end{PROOF}
\subsubsection{Choosing $\Theta$} \label{subsubtheta}
We first discuss a generic methodology for picking the $\Theta$ parameter.  We will consider specific examples below.
Let $x^* = \argmin\{ c(x) \, : \, x \in P \}$, i.e., an optimal solution to the nominal problem \eqref{eq:generic} (without loss of generality, $x^* = x^0$).  Let $0 < \lambda^{\text{lo}} < \lambda^{\text{hi}} \le 1$ and $0 < \xi$. Then we set 
\begin{align} \label{eq:formaltheta}
& \Theta \ = \ \frac{c(x^*) \,  \xi}{\Phi(x^*) \,(\lambda^{\text{hi}} - \lambda^{\text{lo}})}.
\end{align}
\begin{LE}\label{le:formaltheta}
Suppose we run Algorithm \ref{alg:SOFTMAX-adv} 
using \eqref{eq:formaltheta}.  Assume that the 
algorithm does terminate in line 6 at iteration $t$.  
\begin{itemize} 
\item [(i)] If $c(x^t) + \Theta \, \Phi(x^t) \ \le \ c(x^*) + \Theta \,  \lambda^{\mathrm{hi}} \Phi(x^*),$ then 
\begin{align} \label{eq:goodrisk}
 \Phi(x^t) \le \lambda^{\mathrm{hi}} \Phi(x^*) \quad \text{and} \quad c(x^t) \, \le \, c(x^*)\left(1 + \frac{\lambda^{\mathrm{hi}}}{\lambda^{\mathrm{hi}} - \lambda^{\mathrm{lo}}} \xi \right).
\end{align}
\item [(ii)] If  $c(x^t) + \Theta \, \Phi_L^t \ > \ c(x^*) + \Theta \, \lambda^{\mathrm{hi}} \Phi(x^*)$, then 
\begin{align}
    \not \exists x \in P \ \text{with} \ \Phi(x) \le \lambda^{\mathrm{lo}} \Phi(x^*) \ \ \text{and} \ \ c(x) \le (1 + \xi) c(x^*).
\end{align} 
\end{itemize}
\end{LE}  
\noindent \textit{Proof.} (i) Since $c(x^t) \ge c(x^*)$ we obtain the first inequality in \eqref{eq:goodrisk}. Moreover, $\Phi(x^t) \ge 0$ implies
\begin{align} \label{eq:goodcost}
&c(x^t) \, \le \, c(x^*) + \lambda^{\text{hi}} \frac{c(x^*) \,  \xi}{ \lambda^{\text{hi}} - \lambda^{\text{lo}}} \ = \ c(x^*)\left(1 + \frac{\lambda^{\text{hi}}}{\lambda^{\text{hi}} - \lambda^{\text{lo}}} \xi \right).
\end{align}
\noindent (ii) If such an $x$ existed, we would have 
$c(x) + \Theta \,\Phi(x) \le (1 + \xi)c(x^*) + \Theta \lambda^{\text{lo}} \Phi(x^*) \ = \ c(x^*) + \Theta \, \lambda^{\text{hi}}\Phi(x^*).$ 
%$\xi c(x^*) + \Theta \lambda^{\text{lo}}\Phi(x^*) = \xi c(x^*) (1 + \frac{\lambda^{lo}}{\lambda^{hi} - \lambda^{lo}}) = \xi c(x^*) \frac{\lambda^{hi}}{\lambda^{hi} - \lambda^{lo}} = \Theta \lambda^{hi} \Phi(x^*)$
This contradicts the fact that $x^t$ is an optimal solution to the master problem in the last iteration of the algorithm.

\qed
  
\noindent {\bf Remark}. Note that case (i) concerns $\Phi(x^t)$ whereas (ii) uses $\Phi_L^t$ ( $\le \phi_{\max}^t \le \Phi(x^t)$). However, as  per  Lemma \ref{le:phimaxok} and the termination criterion,   $\Phi(x^t)$, $\Phi_{\max}^t$  and $\Phi_{L}^t$ are all "within tolerance" of each other (in many of our numerical examples, the differences are near zero). In other words, the algorithm either generates a solution with both greatly reduced risk exposure and slightly increased cost, or proves that there is no point with moderately tighter requirements.  As an additional point, exact solution of the master problem is only needed for part (ii).\\

\noindent {\bf Example.} Let $\lambda^{\text{lo}} = 0.5$,
$\lambda^{\text{hi}} = 0.6$ and $\xi = 0.01$.  
Under outcome (i) $\Phi(x^t) \le 0.6 \, \Phi(x^*)$ while
also $c(x^t) \le 1.06 \, c(x^*)$.  And under outcome (ii) there is no $x \in P$ with $\Phi(x) \le 0.5 \,\Phi(x^*)$ and
$c(x) \le 1.01 \, c(x^*)$.  \\

\noindent {\bf Discussion points.} \\
\noindent (i) Lemma \ref{le:formaltheta} can be used to develop a number of scaling or binary search algorithms that  compute a
vector with, both, risk exposure smaller than that of $x^*$ by an approximate margin, and cost ramp-up that is likewise approximately constrained; or prove that no 
such vector exists. \\
\noindent (ii) The analysis in Lemma \ref{le:formaltheta} assumes that the algorithm terminates in line 6.  If not, by Theorem \ref{thm:finiteconv}, after a finite number of iterations
we compute a point $\tilde x \in P$ with $\Phi(\tilde x) \le \lambda^{\text{lo}} \Phi(x^*)$.  Should it be the case that $c(\tilde x)$ is only slightly larger than $c(x^*)$ then (as in case (i) of the analysis directly above) we
have computed a point with significantly lower risk exposure and slightly higher cost, as desired. 

What if not, that is to say, what if $c(\tilde x)$ is 
larger than $c(x^*)$ by a margin that is unacceptably high? In that case we appeal to Observation \ref{ob:flip}, i.e., there is \textit{no} $x \in P$ with,
both, $c(x) < c(\tilde x)$ and $\Phi(x) \le \Phi(\tilde x)$.  This is a similar situation as in case (ii) above -- here we have a proof that the goal that we decrease
risk exposure to as much as $\Phi(\tilde x)$ requires
a cost increase deemed excessive.

\subsubsection{Note: the cases $\cZ = \emptyset$ and $|\cZ| = 1$}
In either of these cases the optimization step in line 5 is void.  Hence in line 8 we
simply define $\pi^t_i = \frac{e^{\alpha \phi_i(x^t)}}{\sum_{j \in I} e^{\alpha \phi_j(x^t)}}$ for each $i \in I$.
\subsubsection{Modifications to the algorithm} \label{subsec:modifications}
Here we describe a number of adaptations to Algorithm \ref{alg:SOFTMAX-adv} that we have found effective from an empirical standpoint. In some cases,
those modifications are also backed by theoretical guarantees.

Formally, a boosting kernel will be any algorithm that generates, at an iteration $t$, nonnegative \textit{boosting weights} $w^t_i$ for each $i \in I$, and a separation kernel on the other hand generates coefficients $\pi^t_i = \pi^t_i(w^t)$ with $\sum_{i \in I} \pi^t_i = 1.$  A key point is that, for any separation kernel and any $\bar z \in \cZ$, the 
inequality  $\Phi_L \ \ge \ \sum_{i \in I} \pi^t_i \phi_i(x | \bar z)$ is valid for the risk-adjusted problem \eqref{eq:maxfeatures-epi}.  

Using this language, the boosting kernel in Algorithm \eqref{alg:SOFTMAX-adv} yields
the boosting weights $\phi_i(x^t | z^t)$ where $z^t$ is an approximate solution to the log-sum-exp maximization problem in line 5. The separation kernel applies the softmax function to the quantities $\phi_i(x^t|z^t)$.
However, any combination of a boosting and a separation kernel yields a valid algorithm, and its effectiveness should be judged from an empirical perspective.  We have experimented with several such variants, as follows.\\

\noindent {\bf Greedy method.}  Here the boosting kernel simply selects
$$ (z^t, i^t) \ = \ \argmax_{z \in \cZ, \, i \in I} \phi_i(x^t|z)$$
and assigns weights $w^t_{i^t} = 1$ and $w^t_{i} = 0$ for all other $i$, and the separation kernel sets $\pi^t_{i^t} = 1$ and  $\pi^t_{i} = 0$ otherwise, that is to say the cut added on line 9 is:
$$\Phi_L \ \ge \ \phi_{i^t}(x | z^t).$$
Note that $\pi^t = \text{SOFTMAX}(1,w^t).$   

\begin{LE} \label{le:finiteconvGREEDY}
    Let $\pi\in(0,1)$. Consider Algorithm \ref{alg:SOFTMAX-adv} using the greedy method, and suppose $0 < \lambda < 1$. Either this algorithm terminates finitely in line 6, or we reach an iteration $t$ with $\phi_{\max}^t \le \lambda \phi_{U}^0$.
\end{LE}

\noindent{\bf Comment.} As compared with log-sum-exp, greedy boosting is simpler and can provide adequate performance. However log-sum-exp boosting is in general preferred (over greedy) because each resulting cut incorporates many features. Lemma \ref{le:finiteconvGREEDY} notwithstanding, we have observed, empirically, that log-sum-exp boosting yields faster termination, especially in high-dimensional settings where there are many ties or near-ties among features attaining high $\phi(x^t|z^t)$ value in line 5 at iteration $t$. \\

\noindent {\bf Flattening.} The purpose of the boosting step at an iteration $t$ is to highlight 
larger feature values in the resulting cut deployed in the separation step (i.e., larger $\phi_i(x^t|z^t)$ results in larger $\pi^t_i$). However, while the boosting accentuates order-of-magnitude differences in feature values, we
also want the boosting to be less sensitive to small differences.  Additionally, the numerical task in line 5 may become difficult when feature values are high.  \textit{Flattening}, when deployed, addresses both goals. It replaces each nominal feature value $\phi_i(x^t|z)$ by its logarithm 
$\ln \phi_i(x^t|z)$.\\

\noindent {\bf Synthetic boosting.}  When $\cZ = \emptyset$, Algorithm \ref{alg:SOFTMAX-adv} skips the boosting step altogether.  An alternative 
is to rely on a \textit{synthetic} set  $Z^{\text{synth}} \subset \R^I_+$,
and to define (for example) $\phi_i(x | z) = z_i + \phi_i(x)$.  Hence
the boosting step approximately solves 
\begin{align} \label{eq:synthboost}
 \max_{z \in \cZ^{\text{synth}}} \sum_{i \in I} e^{\alpha(z_i + \phi_i(x^t))}.
\end{align}

\noindent Appropriate examples for $\cZ^{\text{synth}}$ include ``budgets" sets
$\{ z \in \R_+^I \, : \, z_i \le \gamma_i \ \forall i, \, \sum_i z_i \le \Gamma \}$ and 
ball sets $\{ z \in \R_+^I \, : \sum_i z^2_i \le \Gamma\}$. \\

\noindent {\bf Log-barrier term.}  When using synthetic boosting, we have found it 
empirically useful to add, to the objective function \eqref{eq:synthboost} log-barrier terms to prevent the optimal $\zeta_i$ values from reaching an extreme point of the set $\cZ^{\text{synth}}$, i.e., to encourage the solution to the maximization problem to boost multiple features, which will as a result appear in the corresponding cut. 

When, e.g., $\cZ$ is the budgets set described above, the log-barrier contribution added to the boosting problem takes the form $\epsilon \ln(\Gamma - \sum_i z_i) \ + \ \sum_i \epsilon_i\ln(1 - z_i)$ for
positive constants $\epsilon, \, \epsilon_i$ ($i \in I$).\\

\noindent {\bf Clipping.} Depending on the separation kernel, the cuts we obtain may be very dense when $|I|$ is large.  This fact
hinders the solution of the master problem.  Additionally, many of the
cut coefficient may be very small, further challenging solvers.  \textit{Clipping}, at an iteration $t$, works as follows:
\begin{enumerate}
\item Pick a (small) integer $K$.
\item Reset all but the top $K$ values $\pi^t_i$ to zero.
\item The top $K$ $\pi^t_i$ are proportionally reweighed so that they add to $1$.
\end{enumerate}
When $K = 1$, this amounts to using the cut $\Phi_L \ge \phi_{i^t}^t(x|z^t)$, where $i^t = \argmax{ \phi_i(x^t | z^t)}$, which coincides with the greedy cut. 

\subsection{Introductory examples}\label{subsec:introex}
Here we present simple examples where Algorithm \ref{alg:SOFTMAX-adv} or heuristics are applied to optimization problems.  More substantial computations are presented later.

\subsubsection{Flow routing with queueing delays}

Consider a minimum-cost flow problem with a single source
$s$ and a single destination $t$. There are $m$ arcs, and each arc $i$ has a capacity $u_i$ and a per-unit flow cost $c_i$, and there
are $M > 0$ units of flow to route from $s$ to $t$.

In addition, when flow $x_i$ is routed on arc $i$, the \textit{queueing delay} experienced on this arc equals $$ \phi_i(x_i) = \mu(u_i, x_i) \doteq \frac{u_i}{u_i + \epsilon - x_i},$$ where $\epsilon > 0$ (and small) is  used to avoid division by zero\footnote{The function $\mu(u, x)$ is the M/M/1 queuing delay on a channel with capacity $u$ and data rate $x$. See \cite{kleinrock75}}. 

We apply Algorithm \eqref{alg:SOFTMAX-adv} using the cost function $ c(x) \doteq \sum_{i} c_i x_i$ and impact function $\Phi(x) = \max_{i} \phi_i(x_i)$.  (Thus, $\cZ = \emptyset$.)

As an example, consider the network with two nodes (i.e., $s$ and $t$) and three parallel arcs between $s$ and $t$ with data as follows
\[
\begin{array}{r|cccc}
\text{arc}, i & 1 & 2 & 3 & 4\\
\hline
\text{capacity}, \ u_i & 70 & 31 & 50 & 50 \\
\text{per-unit cost}, \ c_i & 1 & 5 & 7 & 8
\end{array}
\] 
There are $M = 100$ units of flow to route. Thus, the nominal optimization to solve is:
\begin{subequations} \label{eq:example1nominal}
\begin{align}
\min & \ x_1 + 5 x_2  + 7 x_3 + 8 x_4  \nonumber \\
\text{s.t.} & \ x_1 + x_2 + x_3 + x_4 \ = \ 100 \nonumber \\
& \ 0 \le x_1 \le 70, \ 0 \le x_2 \le 31, \ 0 \le x_3 \le 50, \ 0 \le x_4 \le 50, \nonumber 
\end{align}
\end{subequations}
with optimal solution $x^0 = (70, 30, 0, 0)$ attaining cost $220$. We set $\epsilon = 0.01$. Thus, 
the maximum queueing delay is attained on arc $1$ and it equals $\mu(70,70) = 100.$

To obtain the master problem we need to add the $\Phi_L$ variable to the nominal optimization problem; additionally, we need to represent the quantity $\phi_i(x_i)$ for each $i = 1, 2, 3$.
This is done by introducing a variable $q_i$ as a proxy for $\phi_i(x_i)$ together with the inequalities
$$ q_i \ \ge \ \mu(u_i, x_i) \quad (= u_i/(u_i + \epsilon - x_i) ), \quad i=1,2,3,$$
which are convex-representable.

Following section \ref{subsubtheta} to set $\Theta$, we choose $\lambda^L = 0.5$, $\lambda^U = 0.6$ and $\xi = 0.01$. Using  $c(x^0) = 220$ and $\Phi(x^0) = 100$, we obtain $\Theta = 0.22$.  
In summary, the master problem is initialized as:
\begin{subequations} \label{eq:master1}
\begin{align}
\min & \ x_1 + 5 x_2  + 7 x_3 + 8 x_4 \ + \ 0.22 \, \Phi_L\nonumber \\
\text{s.t.} & \ x_1 + x_2 + x_3 + x_4 \ = \ 100 \nonumber \\
& \ q_1 \ge \mu(70, x_1), \ q_2 \ge \mu(30, x_2), \ q_3 \ge \mu(50, x_3), \ q_4 \ge \mu(50, x_4) \nonumber\\ 
& \ 0 \le x_1 \le 70, \ 0 \le x_2 \le 31, \ 0 \le x_3 \le 50, \ 0 \le x_4 \le 50, \ \Phi_L \ge 0. \nonumber 
\end{align}
\end{subequations}

Finally, we heuristically set $\alpha = 1/10$.  The algorithm proceeds as follows.  \\

\noindent {\bf 1)} 
Since $\cZ = \emptyset$, line 5 of the algorithm simply amounts to evaluating $\Phi(x^0)$. As per line 8 of the algorithm, the cut we obtain on line 9 is (values rounded to three digits)
$$ \Phi_L \ge  0.980 \, q_1 + 0.018\, q_2, $$
where we have ignored very small (but positive) coefficients on $x_3$ and $x_4$.\\

\noindent {\bf 2)} Resolving the master problem with the added cut yields the solution vector 
$x^1_1 = 69.005$, $x^1_2 = 30.995$, $x^1_3 = x^1_4 = 0$ and $\phi^1_L = 41.266$;  its cost is $c(x^1) = 223.980$. Evaluating queueing delays, we have $\Phi(x^1) = 41.297$.\\

\noindent We note the small gap between $\phi^1_L$ and $\Phi(x^1)$: the formal algorithm is close to termination as per line 6.  However, the central point here is that
\begin{itemize}
\item [(a)] Comparing $\Phi(x^0) = 100$ and $\Phi(x^1) < 41.30$ we see  greatly reduced risk.
\item [(b)] Comparing $c(x^0) = 220$ and $c(x^1) < 223.99$, we see that cost has increased by less than $1.5 \%$.
\end{itemize}
As per our goal of "substantially reducing risk without materially increasing cost" we can terminate the algorithm.  Note that both \eqref{eq:goodrisk} and \eqref{eq:goodcost} are satisfied.

\subsubsection{Min-cost flow under capacity reduction}
Consider the generic single-source, single-destination min-cost flow setup  with $M$ units of flow to route.  Here we consider the case where there is a set $\cZ$ of \textit{decreased capacity vectors}. Each $z \in \cZ$ has an entry $z_{uv}$ corresponding to each arc $(u,v)$,  with value  indicating the capacity of $(u,v)$ in an adverse scenario that is realized after the choice of flow vector $x$ has been made, and which impacts multiple arcs at once.

 Specifically, given a flow vector $x$, we assume that the capacity in each arc $(u,v)$ is (adversarially) reduced to $\min\{x_{uv}, z_{uv}\}$.  As a result,
 some of the $M$ units flow originally routed from $s$ to $t$ using flow vector $x$ may be lost, even if we
 reroute using the reduced capacities.  This model incorporates a form of recourse which we have selected for this example to highlight the behavior of our algorithm.  An alternative would have been to use capacities $z_{uv}$ rather than $\min\{x_{uv},z_{uv}\}$ -- the version
we selected is more constrained in terms of recourse.

\begin{wrapfigure}{r}{0.25\textwidth}
  \begin{center}
    \includegraphics[width=0.25\textwidth]{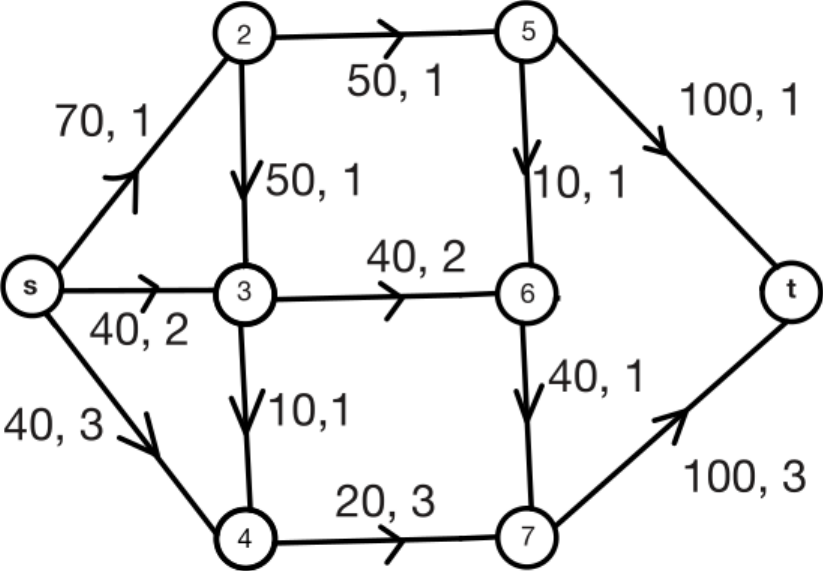}
  \end{center}
  \caption{Reduced capacity example. For each arc we list capacity and per unit cost}
  \label{fig:initial-cap}
\end{wrapfigure}
Formally, given a flow vector $x$ we proceed as follows.  For each vector $z \in \cZ$,
\begin{itemize}
\item [(a)] Let $F(x|z)$ denote the maximum flow that can be routed from $s$ to $t$ in the network where
each arc $(u,v)$ has \textit{effective capacity}  $\min\{x_{uv},z_{uv}\}$. Then $F(x|z) \le M$.
\item[(b)] For each subset of the nodes $S$ with $s \in S$ and
$t \not \in S$ let $\delta^+(S)$ be the set of arcs $(u, v)$ with $u\in S$ and $v \not \in S$.  We write $\text{cap}_S(x|z) = \sum_{(u,v) \in \delta^+(S)}\min\{x_{uv},z_{uv}\} $ to denote the \textit{effective capacity} of the cut, i.e., the cut-capacity using the arc capacities 
in (a). For such a cut define $\phi_S(x|z) = \max\{M - \text{cap}_S(x|z), 0\}$. In our terminology, there is a feature for each cut.

\item [(c)] Let $\cC$ be the set of cuts as in (b). We set $\phi(x|z) \ = \ \max_{S \in \cC} \phi_S(x|z)$.  We have $\phi(x|z) > 0$ if the  effective capacities $\min\{x_{uv},z_{uv}\}$ do not allow $M$ units of flow to be routed from $s$ to $t$.
\end{itemize} 
In summary, we write $\Phi(x) \ = \ \max_{z \in \cZ} \phi(x|z)$ where we assume that  $\cZ$ is compact\footnote{If $\cZ$ is finite and each $z \in \cZ$ has a probability $p_z$ then we could instead consider $\Phi(x) = \sum_{z \in \cZ} p_z \phi(x|z)$. An example is that where we allow a full min-cost flow recomputation once the decreased capacities are realized.}. If $\cZ$ is finite the risk-adjusted problem \eqref{eq:generic-weighted} has an explicit linear representation.  However, if $\cZ$ is large this representation will be prohibitively large both in terms of constraints and variables.

Next we describe our implementation of these ideas in a numerical example, given by the network in Figure \ref{fig:initial-cap}, where we have to route $M = 100$ units of cost from $s$ to $t$.

%\begin{figure}[H]
%        \includegraphics[scale=0.5]{numbers1.pdf}
%        \centering
%        \caption{Reduced capacity example. For each arc we list capacity and per unit cost.}
%        \label{fig:initial-cap}
%\end{figure}
For the sake of simplicity, the reduced capacity model we use is that in each $z \in \cZ$ up to two arcs of the network have their capacity reduced by (exactly) $50 \%$ with no other reductions. For a flow vector $x$ the computation of $\Phi(x)$ can be modeled by a linear mixed-integer program, using binary variables to indicate the arcs where capacity is reduced, and using the dual representation of the min-cut problem -- the \textit{dual MIP} in what follows.

We use the greedy method (Section \ref{subsec:modifications}) and we add only one 
inequality to the master at each iteration.  After initializing the master problem as the min cost flow problem on the given network the algorithm proceeds as follows:\\

\noindent {\bf Iteration 0.  Nominal solution $x^0$.}
\[
\begin{array}{r|cccccccccccc}
\text{arc} & (s,2) & (s,3) & (s,4) & (2,3) & (3,4) & (2,5)& (3,6)& (4,7)& (5,6) & (6,7) & (5,t) & (7,t)\\
\hline
\text{flow} & 60 &40  & 0 & 10  & 10 & 50 & 40 & 10 & 0 & 40 & 58 & 50
\end{array}
\] 
The optimal cost is   $c(x^0) = 560$.   The dual MIP sets $z^0_{s,2} = 35$ and $z^0_{s,3} = 20$ as the two reduced capacities. The cut separating $s$ from the rest of the network 
is a min effective capacity cut, with effective capacity $z^0_{s,2} + z^0_{s,3} + x_{s,4} = 55 + x^0_{s,4} = 55$. Hence $\Phi(x^0) = 100 - 55 =45$.     Hence we add, to the master problem, the cut
$\Phi_L \ + \ 55 \ +\ x_{s,4} \ \ge \ 100.$ 

Also, using $\lambda^{\text{lo}} = 0.7$, $\lambda^{\text{hi}} = 0.75$ and $\xi = 0.01$
\eqref{eq:formaltheta} sets $\Theta = 2.49$. \\

\noindent {\bf Iteration 1 yielding $x^1$. }
\[
\begin{array}{r|cccccccccccc}
\text{arc} & (s,2) & (s,3) & (s,4) & (2,3) & (3,4) & (2,5)& (3,6)& (4,7)& (5,6) & (6,7) & (5,t) & (7,t)\\
\hline
\text{flow} & 70 &10  & 20 & 20  & 0 & 50 & 30 & 20 & 0 & 30 & 50 & 50
\end{array}
\] 
This solution has cost $570$ and it estimates $\Phi_L = 25$. The
dual MIP sets $z^1_{s,2} = 35$ and $z^1_{2,5} = 25$, and the minimum effective-capacity cut 
separates node $s$ from the rest of the network. Its effective capacity equals $z^1_{s,2} + x^1_{s,3} + x^1_{s,4} = 35 + 30 = 65$ and therefore $\Phi(x^1) = 35$.  We add to the master the inequality $\Phi_L + x_{s,3} + x_{s,4} \ge 65.$ \\

\noindent {\bf Iteration 2 yielding $x^2$. }
\[
\begin{array}{r|cccccccccccc}
\text{arc} & (s,2) & (s,3) & (s,4) & (2,3) & (3,4) & (2,5)& (3,6)& (4,7)& (5,6) & (6,7) & (5,t) & (7,t)\\
\hline
\text{flow} & 60 &20  & 20 & 10  & 0 & 50 & 30 & 20 & 0 & 30 & 50 & 50
\end{array}
\] 
The solution to the master problem has cost $570$ and $\Phi_L = 25$.  The min-effective capacity cut now
separates nodes $s, 2, 3, \, \text{and} \, 4$ from the rest, with
$z^2_{2,5} = 25$ and $z^2_{3,6} = 20$, and effective capacity $45 + x^2_{4,7} = 20 = 65$. Hence $\Phi(x^2) = 35$.  Note that the reduced capacity of the adversarial cut is $25 + 20 + x_{4,7} = 45 + x_{4,7}$; however, $x^2_{3,6} = 30$, so $x^2$ only loses $10$ units of flow when
the capacity of $(3,6)$ is reduced to $20$ units (as in the current dual MIP solution).  The inequality we add to the master is $\Phi_L + x_{4,7} \ge 45$.\\

\noindent {\bf Iteration 3 yielding $x^3$. }
\[
\begin{array}{r|cccccccccccc}
\text{arc} & (s,2) & (s,3) & (s,4) & (2,3) & (3,4) & (2,5)& (3,6)& (4,7)& (5,6) & (6,7) & (5,t) & (7,t)\\ 
\hline
\text{flow} & 70 &20  & 10 & 20  & 10 & 50 & 30 & 20 & 0 & 30 & 50 & 50
\end{array}
\] 
The solution to the master has cost $c(x^3) = 570$ and $\Phi_L = 35$.  The 
dual MIP now sets $z^3_{s,2} = 35$ and $z^3_{2,5} = 25$, but only achieves $\Phi(x^3) = 35$ units of lost flow (again using the cut
that separates $s$ from the rest of the network).  Since $\Phi(x^3) = \Phi_L$
the algorithm terminates.\\

\noindent In summary, $$\Phi(x^3)/\Phi(x^0) = 35/45 \approx 0.78 \ \text{and} \  c(x^3)/c(x^0) = 570/560 \approx 1.0179.$$  At the same time, we note
that $c(x^3) + \Theta\, \Phi(x^3) = 570 + 2.49*35 = 657.15 > 644.0375 = c(x^0) + \Theta \, \lambda^{\text{hi}} \, \Phi(x^0)$.  Consequently, by Lemma \ref{le:formaltheta} (ii),
we conclude that 
$$ \not \exists \ \text{feasible flow $x$ with} \ \Phi(x) \le 0.7 \,\Phi(x^0) \ \text{and} \ c(x) \le 1.01 \, c(x^0).$$

\section{Numerical experiments}\label{sec:numericalexper}
Here we describe testing of our algorithms on large problem instances arising in two settings: operation of power grids, and logistics. 

\subsection{Logistics under concentration risk}
In this section we apply our methodology to an evolved version of the logistical setup described in \cite{Kilb26}, which arose from an engagement with a stakeholder, and where a heuristic version of Algorithm \ref{alg:SOFTMAX-adv} was applied.  In brief, the problem being considered has the following attributes.
\begin{itemize}
    \item Multiple commodities are shipped from sources to destinations in a network, using multiple vehicle types.  The different commodities have numerical priorities and delivery time windows with lateness penalties. Vehicles have restricted availability (when and where).
    \item In the numerical instance considered here, there are 44 time periods; the network has 17 nodes and 224 links; there are 94 commodity types and 6 vehicle types.
    \item The resulting MIP formulation has 417,330 variables of which 5,920 are integral and there are 118,343 constraints. Gurobi v12 \cite{gurobi} was used to obtain a solution within $5\%$ of optimality while requiring a few minutes of computation.
    \item It was observed, across multiple instances of the problem, that the computed MIP solution exhibits a high degree of concentration; see Figures \ref{fig:log1} and \ref{fig:log2}.  This feature was deemed (by the stakeholder) to be risk-inducing. In a 
    humanitarian logistics setting, or any setting addressing the delivery of critical items, concentration of resources at a geographical location during a short time span is viewed as risky because the loss of a significant 
    quantity of valuable assets due to unforeseen exogenous events (no matter their nature) is of high impact.
\end{itemize}
We have applied the methodology in this paper as follows.  First, for a solution vector $x$ in the MIP, and for an arc $(i,j)$ and time period $\tau$, we write $\phi_{i,j,\tau} = \phi_{i,j,\tau}(x)$ total weight of all commodities shipped on $(i,j)$ starting at time $\tau$ under solution $x$.  Here, a commodity's "weight" is defined as the product of the numerical priority times the actual shipped weight.  We have also considered, using the same methodology, the case where $\phi_{i,j,\tau}$ only accounts for shipped weight.

For these experiments we used the following version of 
Algorithm \ref{alg:SOFTMAX-adv}:
\begin{enumerate}
\item We set $\Phi(x) = \max_{i,j,\tau} \phi_{i,j,\tau}(x)$, i.e., max shipped  weight, on any link and any point in time.
\item We used flattening, synthetic boosting, and log-barrier terms -- see Section \ref{subsec:modifications} and equation \eqref{eq:synthboost}. Flattening is needed because
the $\phi_{i,j,\tau}$ values can be large, as can be the
number of triples $(i,j,\tau)$, and we used a synthetic set  because there is no overt uncertainty in the model. We relied on
$\cZ^{\text{synth}} = \{ \zeta_{i,j,\tau} \ge 0 \, : \,  \sum_{i,j,\tau} \zeta_{i,j,\tau} \le 1 \}.$ 
Furthermore, we used $\alpha = 1$. Thus, at iteration $t$ of Algorithm \ref{alg:SOFTMAX-adv}, the boosting step (approximately) solves:
\begin{align}\label{eq:logboost}
& \max_{\zeta \in \cZ^{\text{synth}}} \ \left[ \sum_{i,j,t}e^{\zeta_{i,j,\tau}} \phi_{i,j,\tau}(x^t) \quad +\quad \epsilon\ln \left(   1 - \sum_{i,j,\tau} \zeta_{i,j,\tau}  \right) \ + \ \sum_{i,j,\tau} \epsilon_{i,j,\tau} \ln(1 - \zeta_{i,j,\tau}) \right],
\end{align}
where $\epsilon > 0$ and $\epsilon_{i,j,\tau} > 0$ for every
triple $(i,j,\tau)$.
\item This optimization task was addressed using a first-order algorithm detailed below in Section \ref{subsubsec:FOA}.  The algorithm was implemented in multistart fashion, using parallel processing.
\item Each start yields a different vector $\zeta \in \cZ^{\text{synth}}$.  The resulting SOFTMAX cut was added to the master problem if violated.
\item We used $\Theta = c(x^*)/\Phi(x^*)$.  This is consistent with \eqref{eq:formaltheta}
when $\xi = \lambda^{\text{hi}} - \lambda^{\text{lo}}$.
\item Algorithm \ref{alg:SOFTMAX-adv} was run using the LP-relaxation of the MIP, solved using the Barrier method (without crossover) in Gurobi \cite{gurobi}. Upon termination, the MIP, with cuts added, was then run. 
\end{enumerate}

\[
\begin{array}{r|cccc}
\text{Iteration}& \text{Risk} & \text{Cost} & \text{Cuts} & \text{runtime (s)}\\
\hline
0 & 66724 & 1095 & 132 & 25.62  \\ 
1 & 53992 & 1179 & 196   & 28.84  \\ 
2 & 39969 & 1153 & 216   & 30.59  \\ 
3 & 50437 & 1153 & 229   & 30.61 \\ 
4 & 50385 & 1153 & 234   & 30.77 \\ 
5 & 44181 & 1156 &  234  & 29.13 \\
\hline
\text{{\bf MIP}} & 44181 & 1186 &     & 91.65 \\  
\end{array}
\] 
We highlight that the MIP attains a $33.33\%$ decrease in risk metric while incurring an $8.3 \%$ increase in cost.  In fact, in hindsight, we could have terminated the algorithm in iteration 2
where the performance is clearly better both in terms of risk and cost. 
\begin{table}[htbp] 
 \caption{Comparison: cumulative weight shipped in top activities, in units of $10^5$}
 \vskip .1in
 \label{table:bad}
    \setlength{\tabcolsep}{6pt}
    \begin{tabular}{l | rrrrrrrrrr} 
    \text{\# of activities}, K & 1 & 2 & 4 & 8 & 10 & 16 & 32 & 64 & 100 \\
    \hline 
    nominal & $3.45$ & $6.79$ & 12.19 & 20.51 & 23.99 &33.24 & 45.08 & 52.72 & 53.72\\
    risk-aware & $1.45$ & $2.91$ & 5.82 &  11.63 & 14.54 & 23.26 & 38.37 & 50.80 & 54.23 
\end{tabular} 
\end{table}

Table \ref{table:bad} complements Figure \ref{fig:log2}.  It displays the weight shipped in the top $K$ activities for both the nominal and risk-aware solutions, for selected values of $K$ up to $100$, which accounts for nearly all the shipments. We remind the reader that some $150$ activities have positive weight in both solutions (out of, potentially, over $10,000$).  We highlight the high concentration evident in the nominal solution: for example, the top ten activities account for nearly $50\%$ of all
shipped weight.  In contrast, the risk-aware solution reduces the concentration metric
by roughly $50$ for $K < 10$, 
and gradually catches up with the nominal solution -- while incurring a relatively small increase
in cost. 

\subsubsection{First-order method for the boosting step}\label{subsubsec:FOA}

The boosting problem \eqref{eq:logboost} is solved via a custom multi-start first-order ascent algorithm (henceforth termed the \textit{Red} solver) based on AdaDelta \cite{zeiler2012adadelta}, so as to handle the optimization
problem \eqref{eq:logboost}. At the time of this writing, the existing optimization algorithms in Knitro \cite{knitro}, a local nonlinear solver, exhibit difficulty converging on our instances due to the highly nonlinear interaction between the exponential and log-barrier terms; our custom approach exhibits superior performance. A custom multi-start, first-order algorithm based on AdaDelta proves superior to other popular first-order algorithms for its ability to self-calibrate the sensitive choice of step size, and because it does not use momentum which often pushes iterates past the log-barrier. Multi-starting generates multiple runs per boosting step. We provide further details next.

The feature values, $\phi_{i,j,\tau}(x^t)$, of the current solution are normalized and scaled by a predetermined constant to prevent either the boosting term or log-barrier terms from dominating the objective. We mutli-start the first-order solver with a number of initial points, $n_\text{starts}$. For each of these starts $1 \le h \le n_\text{starts}$ we select an initial point $\zeta^{(h)}$ using a sliding window, $\mathcal{W}_h$, of given size and step, on the sorted indices of the feature values $\phi_{i,j,\tau}(x^t)$ via,
\begin{align*}
    \zeta^{(h)}_{i,j,\tau} = \begin{cases} 
      w & (i,j,\tau) \in \mathcal{W}_h \\
      0 & (i,j,\tau) \notin \mathcal{W}_h
\end{cases},
\end{align*}
where $w > 0$ is some small constant.

The motivation for using a sliding window is twofold. Firstly, in high dimensions, the sum $\sum_{i,j,\tau} \zeta_{i,j,\tau}^{(h)}$ easily surpasses the the budget of $1$, i.e., violates the aggregate log-barrier, even for near-zero efforts across each feature. So we only initialize a subset of variables as nonzero; namely, those corresponding to the largest feature values. Secondly, the largest value $\zeta_{i,j,\tau}^{(h)}$ in a solution, likely corresponds to the largest feature value $\phi_{i,j,\tau}(x^t)$, which, if only the cut associated with $\zeta^{(h)}$ is added, can cause the next solution, $x^{t+1}$, to shift weight away from the tuple $(i,j,\tau)$. The sliding window masks larger features from previous starts in order to elucidate the next greatest exposure to concentration in a look-ahead-type fashion.

Optionally, indices of $\zeta^{(h)}$ corresponding to the largest feature values within the window are set to a larger constant $w'>w$. This mitigates the solver converging to local optima where equal weight is put on multiple features, when an asymmetric solution focusing on the features with higher concentration is superior. Moreover, during each start $h$, the gradients of coordinates $\zeta^{(h)}_{i,j,\tau}$ outside of the window are filtered out of the first-order update computation, thus fixing these coordinates to $0$ during the course of the solve.

Worker processes are used to parallelize the solve of each initial point via a custom implementation of AdaDelta using projection and backtracking. Projection is used to prevent variables from becoming negative, and is implemented by clipping, i.e., taking the positive part of the iterate $[\zeta^{(h)}]^+$, after each gradient step. Conversely, backtracking is employed if a variable violates the log-barrier, and is implemented so as to satisfy Wolfe conditions \cite{nocedal2006numerical}.

The disparity in which corrective measure to use for maintaining feasibility is based on our belief that $\zeta$ variables approaching the log-barriers indicate features worth disrupting and so merit a delicate analysis, whereas those approaching $0$ are unappealing to the solver and can be discarded.

We now specify the parameters and settings for the Red solver used in our experiments.
The log-barrier constants are $\epsilon = \epsilon_{i,j,\tau} = 1.0$, for the simple reason that the log-barriers tended to dominate the objective for larger values. With this choice of parameters, normalizing and scaling the features of the current solution, $\phi_{i,j,\tau}(x^t)$, by $10.0$ provides a balanced interaction between the boosting and log-barrier terms in the objective.

For multistarting we used $h^\text{starts}=30$, a window of size $10$ and step $1$, and value $w=1/(10\dim W)$. A window size of $10$ enhances the first-order solver focus on the highest concentration for each start. The constant $\frac{1}{10\dim W}$ induces the initial sum of $\zeta$ variables, $\frac{1}{\dim W}$, to be much less than the aggregate budget $1$. Moreover, we warmstarted $1$ coordinate (the largest feature in the window) to the value $w'=0.5$. Finally, the internal parameters for AdaDelta and backtracking are the default values in their respective literature  \cite{zeiler2012adadelta, nocedal2006numerical}.

% \begin{table}[H]
% \centering
% \caption{Experimental settings for initial point selection for the Red Solver}
% \label{tab:red_implementation}
% \begin{tabular}{l c}
% \hline
% Setting & Value \\
% \hline
% $h^\text{starts}$ & $30$ \\
% window size & $10$ \\
% window step & $1$ \\
% warm start size & 1\\
% $w$ & $\frac{1}{10 \dim W}$ \\
% $w'$ & $0.5$ \\
% \hline
% \end{tabular}
% \end{table} 

\subsection{Constraint coefficient uncertainty}
Now consider problems of the form (\ref{eq:nonlineargeneric}) where the constraint coefficients, $a_{ij}$, are subject to uncertainty. Such uncertainties can arise in problems whose formulations rely on physical measurements, such as in power flow \cite{cain+etal12}, pooling \cite{pooling}, and beamforming \cite{beamforming}.  See \cite{blake2} for a more extensive analysis.

For experimentation, we consider instance QPLIB 2894 of the Quadratic Programming Library \cite{qplib}, which is a linear quadratically constrained program with 17 variables and 209 constraints. Many constraints, both linear and quadratic, feature fractional coefficients that could contain errors. For example, see constraint e198 (with coefficients rounded to nearest thousandth) below.
\begin{align*}
&22.009\, x_6^2 + 9.132\, x_6x_8 + 18.458\, x_6x_9 - 11.482\, x_6x_{10} - 9.131\, x_6x_{11}- 2.156\, x_6x_{16} + 2.156\, x_6x_{17} \\ 
&+ 22.009\, x_7^2 + 2.156\, x_7x_8 + 11.482\, x_7x_9+ 18.458\, x_7x_{10} - 2.156\, x_7x_{11} + 9.132\, x_7x_{16} - 9.132\, x_7x_{17}+\, x_8^2 \\ 
&+ 4.391\, x_8x_9 - 1.478\, x_8x_{10} - 2\, x_8x_{11} + 5.367\, x_9^2 - 4.391\, x_9x_{11} + 1.478\, x_9x_{16} - 1.478\, x_9x_{17}+ 5.367\, x_{10}^2\\ 
& + 1.478\, x_{10}x_{11} + 4.391\, x_{10}x_{16} - 4.391\, x_{10}x_{17} +\, x_{11}^2 +\, x_{16}^2 - 2\, x_{16}x_{17}+\, x_{17}^2 \ \ge \ 1.300
\end{align*}
At the optimal solution $x^*$, $x^*_7 = 1.000$ and $x^*_{17}=8.450$, so the term $- 9.132\, x^*_7x^*_{17}=-77.165$. Moreover, this constraint has a slack of 0.024. Should the coefficient $-9.132$ of the monomial $x_7x_{17}$ decrease by $1\%$ to $-9.223$, the value of the monomial would become $-77.934$, and $x^*$ would violate the constraint by $0.745$ or $57.5\%$ of the right-hand side. In fact, any error in this coefficient beyond $0.04\%$ in the negative direction results in a violation.

To measure the adversarial infeasibility incurred by x feasible for (\ref{eq:nonlineargeneric}), we use
\begin{align*}
\phi_i(x|z) \ = \ \frac{1}{s_i}\left[ \sum_{j \in J_i} a_{ij}f_{ij}(x) (1+z_{ij}) - b_i \right]^+ \ \doteq \ \frac{1}{s_i}\left[ \sum_{j \in J_i} a_{ij}f_{ij}(x) z_{ij} - \text{slack}_i(x) \right]^+,
\end{align*}
where $z_{ij}\geq0$ is the error on coefficient $a_{ij}$, $s_i \geq 0$ is some scale factor for constraint $i$, and $\text{slack}_i(x) = b_i-\sum_{j\in J_i}a_{ij}f_{ij}(x)$ is the slack of constraint $i$ for solution $x$. We examine the two scaling choices: $s_i = \|a_i\|_\infty$ for ``scaled absolute violation" and $s_i = \max(b_i, 1)$ for ``percent violation". For the adversarial budget we used the Euclidean ball $\cZ = \{ z \in \R^{d} \, : \, \|z\|_2 \le .01\}$, where $d= |I|\sum_{i\in I}|J_i|$, hence limiting individual coefficient error to within $1\%$ in the worst case. Finally, we used the following implementation of Algorithm \ref{alg:SOFTMAX-adv}:
\begin{itemize}
    \item We compute parameter $\Theta$ as per \eqref{eq:formaltheta}, with $\xi = 0.01$, $\lambda^{\text{lo}} = 0.48$ and $\lambda^{\text{hi}} =0.50$.
    \item The boosting step used the greedy method, i.e., at iteration $t$ we set
    $$z^t \ = \ \argmax_{z \in \cZ\,:\,\|z\|_0 = 1} \max_{i \in I} \  \frac{1}{s_i}\left[ \sum_{j \in J_i} \, a_{ij}f_{ij}(x) z_{ij} - \text{slack}_i(x) \right]^+,$$% \ \doteq \ .01e_{\hat i\hat j}\, ,$$
    which is the unit vector $.01e_{i^tj^t}$ for the coordinate $(i^t,j^t)\in I\times J_i$ that maximizes $\frac{1}{s_i}\left[ .01(a_{ij}f_{ij}(x)) - \text{slack}_i(x) \right]^+$.
    \item For the cutting step we add the cut $\Phi_L \geq \phi_{i^t}(x|z^t) = \left[ .01(a_{i^tj^t}f_{i^tj^t}(x)) - \text{slack}_{i^t}(x) \right]^+ / s_{i^t}$.
    \item The algorithm is run until it terminates as per criterion $b$ on line $6$.
    \item The master problem is solved using Gurobi \cite{gurobi}.
\end{itemize}

We present results for both measures of adversarial infeasibility in tables \ref{table:scaledabsviolation} and \ref{table:percentviolation}. Revisiting the monomial $x_7x_{17}$ in constraint e198, the potential $57.5\%$ violation is the initial worst-case for the percent violation measure of infeasibility. Upon termination of the algorithm, the output, $\hat x$, has $\hat x_7=0.998$ and $\hat x_{17}=8.472$, so the term $-9.132\hat x_7 \hat x_{17} = -77.212$. However, this constraint now has a slack of 0.717, so a $1\%$ decrease in the coefficient $-9.132$ only results in violation of $4.03\%$ of the right-hand side. Moreover, any error in this coefficient smaller than $0.92\%$ will not result in a violation of constraint e198 for $\hat x$.

\begin{table}[htbp] 
 \caption{Results for scaled absolute violation}
 \vskip .1in
 \label{table:scaledabsviolation}
    \setlength{\tabcolsep}{6pt}
    \begin{tabular}{c | cccc} 
    \text{iterations} &  \text{runtime (s)} &\text{worst-case scaled abs. violation} & \text{corresp. percent violation, $\%$} & \text{cost $\Delta$, $\%$}\\ 
    \hline 
    0 & 2.76  & 0.069 & 4.80 & -\\
    2 & 11.96 & 0.034 & 12.20 & 1.00\\
    7 & 32.65 & 0.014 & 4.20 & 1.06
\end{tabular} 
\end{table}

\begin{table}[htbp] 
 \caption{Results for percent violation}
 \vskip .1in
 \label{table:percentviolation}
    \setlength{\tabcolsep}{6pt}
    \begin{tabular}{c | cccc} 
    \text{iterations} &  \text{runtime (s)} &\text{worst-case percent violation ,$\%$} & \text{corresp. scaled abs. violation} & \text{cost $\Delta$, $\%$}\\ 
    \hline 
    0 & 2.78  & 57.53 & 0.034 & -\\
    1 & 7.38 & 17.85 & 0.045 & 1.07\\
    6 & 19.67 & 4.23 & 0.002 & 1.08
\end{tabular} 
\end{table}

\subsection{ACOPF under large branch loading and thermal risk} 

The ACOPF (Alternating Current Optimal Power Flow) problem arises in the operation of power grids, and has received considerable attention in part due to the increasing importance of efficient energy delivery, and also because of a perceived increase in failures of grids. To put the risk issue into perspective, a standard operating requirement for power grids is that they be able to tolerate the failure of any one component (for any reason whatsoever) and this requirement is incorporated into optimization algorithms in daily use; the so-called N-1 rule.  

A more evolved risk attitude
focuses on a variety of  flavors, including thermal management issues, and we address that perspective in this section.  A salient fact is that imposing a 
strict "N-K" requirement for $K > 1$ but small (e.g., $K = 3$) would be onerous, and an alternative is to reduce rather than eliminate risk, if it can be done at small cost.

For background on ACOPF and related issues, see \cite{carpentier62}, \cite{cain+etal12}, \cite{bienstock15}, \cite{bienstock+etal20}, \cite{molzahn+hiskens19}, \cite{coffrin+vanhentenryck14}.  For the purpose of explaining the experiments described here, we provide a brief outline of the problem and its risk implications, using a mix of power systems and generic terminology.

\begin{itemize}
    \item The problem seeks to generate power at minimum cost, subject to satisfying a set of demands, and physical constraints.  
    \item As input to the problem there is a network (in the graph theoretic sense) whose \textit{branches} (i.e., arcs) model transmission lines and are endowed with numerical parameters that describe the power flow physics. 
    \item Power is generated at the nodes of the network, where generators are housed; each generator is described by a (convex) cost function and a maximum output.  Demands are also situated at
    the nodes of the network.
    \item Power flows on branches are described by nonlinear and nonconvex equations involving physical quantities (voltages) which are also variables and are nodal attributes.
    \item For each branch $(k,m)$ the power flowing from $k$ to $m$ is 
    indicated by a variable $P_{km}$\footnote{A technical detail is that, in general, $P_{km} \neq -P_{mk}$ though the difference is small.}.
    \item Each branch $(k,m)$ has a \textit{resistance} $r_{km} > 0$, and the thermal energy radiated by the line is, approximately, $r_{km} \, P^2_{km}$ under standard assumptions on voltages (a more accurate representation is provided by the \textit{current} on branch
    $(k,m)$).
    \item There is a standard approximation (not a relaxation) to ACOPF that is linear: the so-called DC approximation, or DCOPF in short, which is in fact the standard formulation used in economic grid operation (even if ACOPF is more accurate and used primarily for generic risk assessments).  DCOPF does include quantities playing the role of the $P_{km}$ above, and has identical cost function as
    ACOPF.
\end{itemize}
In the experiments that we describe in this section, cost (i.e., the $c(x)$ in our formulations) will be the generation cost.   We will use
$x$  to generically refer to the set of all variables in ACOPF.  To describe the impact function $\Phi(x)$ we rely on setups using the $r_{km} \, P^2_{km}$ thermal quantities described above. In what follows we write
$$T_{km} = T_{km}(x) \doteq r_{km} \, P^2_{km}$$
for brevity. Also let $\myB$ denote the set of branches of the network.
Here we describe four different experiments, using a combination of $\Phi(x)$ functions, without and with a set $\cZ$, and using different boosting and separation approaches.  Some common details are as follows:
\begin{enumerate}
\item In all cases we compute the quantity $\Theta$ as per \eqref{eq:formaltheta}: $ \Theta \ = \ \frac{c(x^*) \,  \xi}{\Phi(x^*) \,(\lambda^{\text{hi}} - \lambda^{\text{lo}})}$, with $\xi = 0.01$, $\lambda^{\text{lo}} = 0.4$ and $\lambda^{\text{hi}} =0.5$.
\item To set $\alpha$ in Algorithm \ref{alg:SOFTMAX-adv} we used the following heuristic (which is aligned with the theory presented above).
Let $T^* = \max_{km}T_{km}$ in the nominal case.  Then $\alpha = \min\{50, \ln|\myB|/(0.25 \Phi(x^*))\}.$  The denominator in this last expression was motivated by empirical observations using relevant examples.
\item The algorithm is run until it terminates, or it achieves $\Phi_L \le 0.5 \, \Phi(x^0)$.
\item The algorithm is run using the DCOPF formulation which is addressed using Gurobi.  Upon termination, the cuts are transferred to the ACOPF formulation and we run Knitro \cite{knitro} using the amplpy \cite{amplpy} interface to handle that problem.  We also run Knitro on the nominal ACOPF formulation in order to obtain the nominal cost and risk values.
\item With regards to the item above, an important perspective here is that while the AC formulation is clearly better aligned with power flow physics, the energy markets actually rely on the DC formulation.  In light of that fact the DC-based approach is in a precise sense more meaningful, but we include the AC-based analysis for completeness.
\end{enumerate}
In the experiments described below, we used the following well-known cases present in the Matpower \cite{Matpower} or pglib \cite{pglib_opf} libraries:
\[
\begin{array}{r|lccccccccccc}
\text{\#} & \text{name} & \text{nodes} & \text{branches} & \text{generators}  &   \\ 
\hline
1 &\text{pglib$\_$opf$\_$case13659$\_$pegase$\_\_$api.m} & 13659 & 20467& 4092  \\
2 &\text{pglib$\_$opf$\_$case30000$\_$goc$\_\_$api.m} & 30000 & 35393 & 
3526 \end{array}
\] 
These cases have the following attributes:

\[
\begin{array}{r|cccccccccccc}
\text{\#} &  \text{variables, AC} &  \text{constraints, AC} &\text{variables, DC} & \text{constraints, DC}\\ 
\hline
1 &   137837 & 170587  & 51878 &  47785\\
2 & 222878 & 307751 & 98920 & 95393
\end{array}
\] 

\subsubsection{Introductory case with $\cZ = \emptyset$ and maximum thermal metric} \label{subsubsec:maxthermal}
 First we
use
$$ \Phi(x) \ = \ \max_{(k,m) \in \myB} T_{km}.$$
Thus, there is a feature per branch, and since $\cZ = \emptyset$ the boosting step (line 5 of Algorithm \ref{alg:SOFTMAX-adv}) is void. In terms of the cut used in the algorithms, we define $$ \pi_{km} \ = \ \frac{e^{\alpha T_{km}}}{\sum_{(k',m') \in \myB} e^{\alpha T_{k'm'}}},$$ that is to say, $\pi = \text{SOFTMAX}(1,T)$.  The coefficients $\pi_{km}$ are rounded to zero if they fall beneath $10^{-6}$.   We obtained the following results; in all tables below  "cost $\Delta$" stands for cost increment.
\[
\begin{array}{r|cccccccccccc}
\text{\#} &  \text{iterations} &  \text{runtime (s)} &\text{risk reduction, $\%$} & \text{cost $\Delta$, $\%$} & \text{AC risk reduction, $\%$} & \text{AC cost $\Delta$, $\%$}\\ 
\hline
1 &   4 & 5.63  & 50.18 &  0.56 & 31.91 & 0.44 \\
2 & 3 & 6.6 & 47.86 & 0.80 & 28.64 & 0.48 %run2
\end{array}
\] 

Note the relative degradation of the improvement in the risk metric when we go from the DC to the AC formulation.  

\subsubsection{Max thermal metric under distributionally robust exogenous risk} \label{subsub:maxthermalrob}
Next we consider a variant of the above max thermal feature with a significant difference.
Namely, we assume
that the thermal metric $T_{km} = r_{km} P^2_{km}$ is now augmented with a term that 
reflects an exogenous source of risk.  Briefly, elevated transmission line temperatures 
pose a number of risks that are difficult to quantify.  Additionally, line
temperatures are due to a combination of endogenous factors (our $T_{km}$ term, in approximation) and 
exogenous and uncertain factors that are extremely difficult to precisely specify and measure.  See the 
IEEE 738 standard \cite{IEEE738-2012}.  

Both factors, endogenous and exogenous, are nominally combined 
using the heat equation, whose (deterministic) solution would yield line temperature as a function of time -- once more, an effectively impossible task due to data unavailability to the operator.  As a final point, it is worth stressing that under such a model, actual line temperature on an branch $(k,m)$ is a highly nonlinear function of $T_{km}$ and the exogenous factors.

We take a stance where we measure actual risk by adding, to the $T_{km}$ term, a second term that reflects a given multiple of a standard deviation of additional thermal risk exposure.  We assume that to first order, that additional term is \textit{proportional} to $T_{km}$.
 
 More precisely, for a  branch $(k,m)$,  we write $\phi_{km}(x|z) \, = \, (1 + z_{km}) T_{km}$, where $z_{k,m} \ge 0$ is the multiplicative exogenous impact on risk.  In this section we focus on the case where
 $$ \Phi(x) \ = \ \max_{z \in \cZ} \max_{(k,m) \in \myB} \phi_{km}(x|z),$$
and we will use a budgets set
$$ \cZ \ = \ \left\{ 0 \le z_{k,m} \le 1 \, : \, \sum_{(k,m) \in \myB} z_{k,m} \le N \right\},$$
for some $N > 0$. In terms of the usage of Algorithm \ref{alg:SOFTMAX-adv}, we applied the following rules:
\begin{itemize}
\item The solution to the boosting problem in line 5 is obtained by setting $z_{km} = 1$ for the
$N$ branches $(k,m)$ with top value $T_{k,m}$, and $z_{k,m} = 0$ otherwise.  
\item We used $N = 5$, and we applied \textit{clipping} (Section \ref{subsec:modifications})  
when computing the cut used in the separation step.  This last feature is motivated by the fact that
at any given iteration $t$, the branches with $z^t_{k,m} = 1$ are also those with largest value $T_{k,m}$, and, as a result, the sum of the top  $2N$ values $e^{\alpha (1 + z_{km}) T_{km}}$ 
strongly dominates the sum of all such terms.  In other words, the top $2N$ values $\pi^t_{km}$ 
already add up to approximately $1$ (and the rest contribute a negligible amount).

\end{itemize}

\[
\begin{array}{r|cccccccccccc}
\text{\#} &  \text{iterations} &  \text{runtime (s)} &\text{risk reduction, $\%$} & \text{cost $\Delta$, $\%$} & \text{AC risk reduction, $\%$} & \text{AC cost $\Delta$, $\%$}\\ 
\hline
1 &  10 & 18.78 & 50.76 & 1.05 & 36.41 & 0.56 \\
2 & 5 & 64.66 & 50.33 & 0.56 & 28.43 & 0.48 %run13
\end{array}
\] 
\subsubsection{Average of top ten}\label{subsubsec:avtop10}
Let $T_{(i)} = T_{(i)}(x)$ be the $i^{th}$ largest value $T_{km}$. We remind the reader that the $T_{km} = r_{km} P^2_{km}$ and that $x$ is the vector of all variables, including the $P_{km}$.  In the next 
set of experiments we use
$$ \Phi(x) \ = \ \frac{1}{10} \sum_{i = 1}^{10} T_{(i)}.$$
This is again a case with $\cZ = \emptyset$. If we use
$\myB^{10}$ to denote the set of all 10-tuples chosen from $\myB$, we can equivalently write
\begin{align} \label{eq:phi5}
& \Phi(x) \ = \ \max_{s \in \myB^{10}} \phi_s(x), \quad \text{where for $s \in \myB^{10}$ we write} \ \phi_s(x)=\frac{1}{10}\sum_{(k,m) \in s} T_{km}.
\end{align}
In other words, there is a feature for each 10-tuple of branches.   For the separation step 
we used the following approach:
\begin{itemize}
\item At any iteration $t$ we generate a set $S^t \subset \myB^{10}$.  
\item We construct $S^t$ by selecting 10-tuples chosen from among the top 20 branches $(k,m)$ as per the metric $T_{km}$.
The set $S^t$ is guaranteed to include the tuple $s$ attaining maximum
value $\phi_s(x^t)$ as in \eqref{eq:phi5}.
\item For each tuple $s \in S^t$ we add the cut $\Phi_L \ge \ \frac{1}{10} \sum_{(k,m) \in s} T_{km}$.  Thus, with $|S^t| = 10$, ten cuts are added in each iteration.
\end{itemize}

\[
\begin{array}{r|cccccccccccc}
\text{\#} &  \text{iterations} &  \text{runtime (s)} &\text{risk reduction, $\%$} & \text{cost $\Delta$, $\%$} & \text{AC risk reduction, $\%$} & \text{AC cost $\Delta$, $\%$}\\ 
\hline
1 &   5 & 14.48  & 51.71 & 1.41 & 33.96 & 1.01 \\
2 & 10 & 47.46 & 40.12 & 0.66 & 20.63 & 0.60 %run30kave10
\end{array}
\]

%\[
%\begin{array}{r|cccccccccccc}
%\text{\#} &  \text{iterations} &  \text{runtime (s)} &\text{risk reduction, $\%$} & \text{cost $\Delta$, $\%$} & \text{AC risk reduction, %$\%$} & \text{AC cost $\Delta$, $\%$}\\ 
%\hline
%1 &   4 & 11.22  & 50.54 & 1.16 & 38.69 & 0.87 \\
%2 & 10 & 39.93 & 43.22 & 0.88 & 29.11 & 0.91 %run5
%\end{array}
%\] 

\subsubsection{Average of top five under distributionally robust exogenous risk} \label{subsubsec:top5z}
Here we continue with the model above, but, as in Section \ref{subsub:maxthermalrob}, we use
a multiplicative model for endogenous risk. More precisely, for a tuple
$s \in \myB^5$ we write
$$ \phi_s(x|z)  \ = \ \frac{1}{5} \left( \sum_{(k,m) \in s}(1 + z_{km}) T_{km} \right),$$
where $z_{km} \ge 0$ is the multiplicative exogenous impact on $(k,m)$.  Continuing as above, we will
now focus on 
$$ \Phi(x) \ = \ \max_{z \in \cZ} \max_{s \in \myB^5} \   \phi_s(x|z) \ = \ \max_{z \in \cZ} \max_{s \in \myB^5} \ \frac{1}{5}\left(\sum_{(k,m) \in s} (1 + z_{km}) T_{k,m}\right).$$
For these 
experiments we used
$$ \cZ \ = \ \{ z \in \R^{\myB} \, : \, \|z\|_2 \le 1\},$$
i.e., the Euclidean unit ball.   In terms of implementation of Algorithm \ref{alg:SOFTMAX-adv} we used the following methodology:
\begin{itemize}
\item The boosting step used the greedy method, that is to say, at iteration $t$ we
set 
$$z^t \ = \ \argmax_{\|z\|_2 = 1} \max_{s \in \myB^5} \  \sum_{(k,m) \in s}   \, z_{km} T_{km}  \ = \ \argmax_{\|z\|_2 = 1}  \sum_{(k,m) \in s^*}   \, z_{km} T_{km},$$
where $s^* \in \myB^5$ is the ordered 5-tuple made up of the top 5 branches under the $T_{km}$ metric.  A computation shows that 
$$z^t_{k,m} = \frac{T_{km}}{\sqrt{\sum_{(k',m') \in s^*} (T^t_{k',m'})^2}} \ \text{for} \ (k,m) \in s^*, \quad z^t_{k,m} = 0 \ \text{otherwise}.$$
\item For the cutting step we used the same approach as in the above section, i.e., we select a set of tuples $s$ (including $s^*$) (a total of ten tuples in the set) and for each tuple we add the 
cut $\Phi_L \ge \frac{1}{5} \sum_{(k,m) \in s} (1 + z^t_{km}) T_{km}$.
\end{itemize}

\[
\begin{array}{r|cccccccccccc}
\text{\#} &  \text{iterations} &  \text{runtime (s)} &\text{risk reduction, $\%$} & \text{cost $\Delta$, $\%$} & \text{AC risk reduction, $\%$} & \text{AC cost $\Delta$, $\%$}\\ 
\hline
1 &  10 & 11.92  & 56.78 & 1.34  & 36.26 & 1.03 \\
2 & 12 & 62.66 & 43.39 & 0.88& 27.20 & 0.81 %run30kave10exo
\end{array}
\] 

\subsubsection{Analysis of outcomes}
Here we compare the solutions computed using the above models, as specific to case \#2 (the 30,000-node system).

In Figure \ref{fig:thermal} we display the top 30 values $T_{km}$ for the solution computed in the nominal case, the max-robust case in Section \ref{subsub:maxthermalrob}, the max-top-10-average in Section \ref{subsubsec:avtop10} and the average-top-5-robust case considered in Section \ref{subsubsec:top5z}.  It is important to note that the 30 values plotted on each of these four
curves refer, in each case, to a set of 30

As expected, all solutions produced by the algorithm substantially
reduce the risk present in the nominal solution.  The max-robust solution is the least attractive among those but even so, it continues to de-risk the nominal solution even after we move past the maximum.

Another feature of interest is that these three solutions have similar $\max_{km} T_{km}$ metric. However  the "max-robust solution" appears inferior to the "average of top 5 (robust)" and the 
"average of top 10" solutions.  The latter, superficially, appears
to be the most appealing, but it should be remembered that all
three of the "robust" solutions have that additional attribute, i.e., they are all protected against (a model of) exogenous risk.

An additional perspective is gained by comparing how the different solutions arrange power 
flows.  The table below displays the top ten branches in terms of the thermal metric $T_{k,m}$ for the nominal solution:
\[
\begin{array}{r|rrrrrrrrrrcc}
\text{rank} &  1 & 2 &3 & 4 & 5 & 6 & 7 & 8 & 9 & 10\\ 
\hline
\text{branch} &  30172 & 3170  & 1408 & 5295  & 1412 & 5296 & 5463 & 2741 &5287 & 17813\\
T_{km} & 0.62 & 0.46 & 0.40 & 0.37& 0.34 & 0.34 & 0.31 & 0.30 & 0.30 & 0.27 \\
P_{km} & 11.59 & 38.70 & -37.23 & -37.89 & -32.12 & -33.79 & -28.44 & 33.64 & 28.43 & 8.63
\end{array}
\] 
We can compare these values to the outcome from the "average of top five, robust" case, where we highlight the branches that do not appear in the table above:
\[
\begin{array}{r|rrrrrrrrrrcc}
\text{rank} &  1 & 2 &3 & 4 & 5 & 6 & 7 & 8 & 9 & 10\\ 
\hline
\text{branch} &  1408 & 17813  & 30172 & {\bf 28795}  & 3170 & 2741 & {\bf 21388} & {\bf 7632} & 1412 & {\bf 17835}\\
T_{km} &  0.33 & 0.26 & 0.22 & 0.22& 0.22 & 0.22 & 0.22 & 0.21 & 0.21 & 0.21 \\
P_{km} &  -33.53 &  8.51 & 6.86 & -8.61 & 27.22 &28.30 & 16.83 & 32.45 & -25.28 & -21.53
\end{array}
\] 
\textit{different} branches $(k,m)$.  

\begin{figure}[H]
        \includegraphics[scale=0.6]{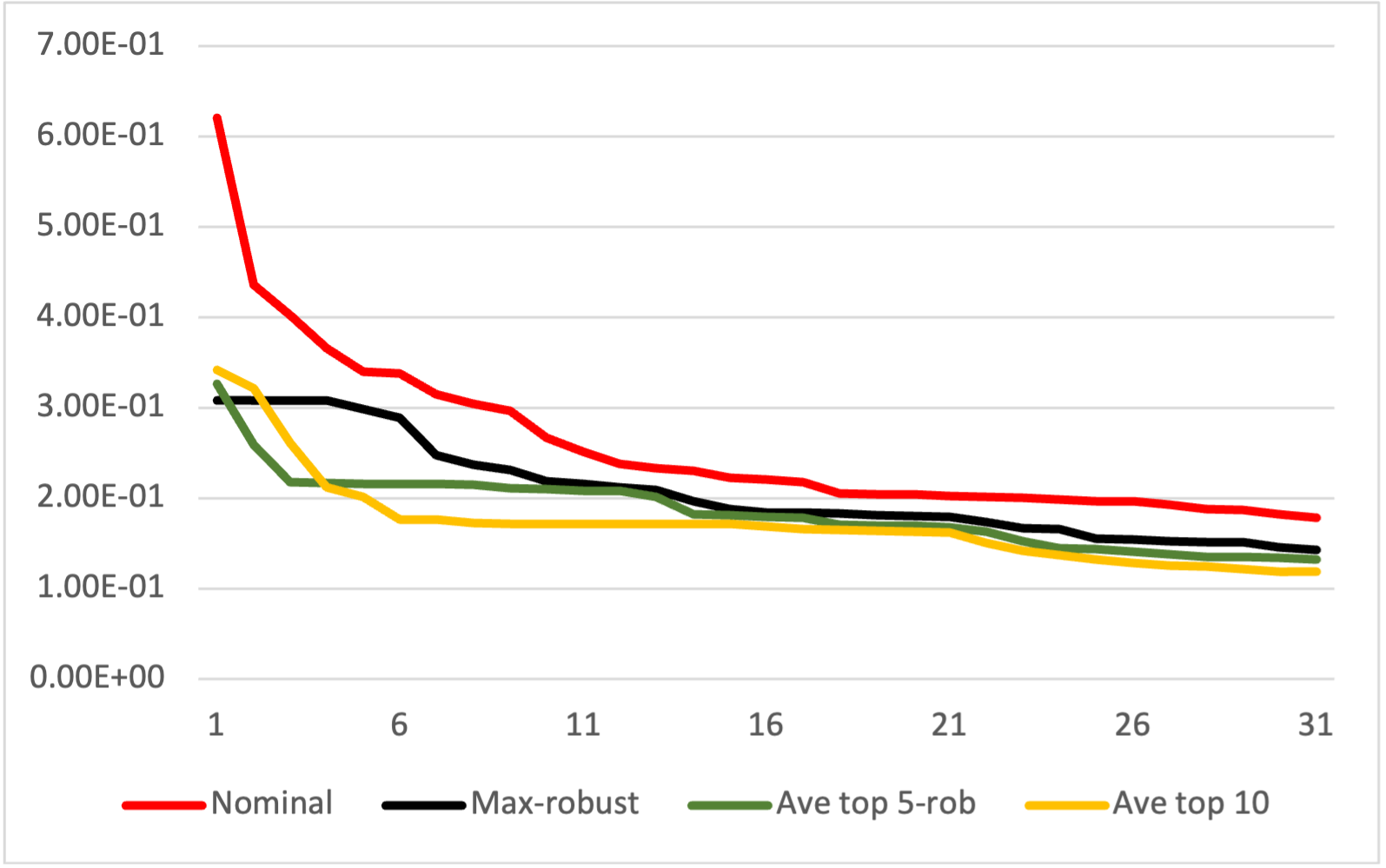}
        \centering
        \caption{Thermal metrics comparison}
        \label{fig:thermal}
\end{figure}

%\begin{wrapfigure}{r}{0.5\textwidth}
%  \begin{center}
%    \includegraphics[width=0.48\textwidth]{thermals.png}
%  \end{center}
%  \caption{Thermal metrics comparison }
%  \label{fig:thermal}
%\end{wrapfigure}

Note that the top six $T_{km}$ entries in the table depicting the nominal solution are 
all higher than the maximum for the robust solution (i.e., $0.33$) even though that solution only considers the
average among the top five.  Moreover, branch 17813 is the tenth-ranked in the nominal solution yet it attains a metric that is higher than all but the highest branch in the robust solution.    

\bibliographystyle{splncs04.bst} 
\bibliography{redblue}

@misc{gurobi,
  author = {{Gurobi Optimization, LLC}},
  title = {{Gurobi Optimizer Reference Manual}},
  year = 2026,
  url = "https://www.gurobi.com"
}

@article{Matpower,
    author = {Zimmerman, R.D. and Murillo-Sanchez, C.E. and Gan, D.},
    title = {{MATPOWER, A MATLAB Power System Simulation Package}},
    journal = {IEEE Trans. Power Sys.},
    year = {2011},
    pages = {12-19},
    volume ={26}
}

@article{markowitz,
    author = {Markowitz, H},
    title = {{Portfolio Selection}},
    journal = {The Journal of Finance},
    year = {1952},
    pages = {77-91},
    volume ={7}
}

@article{pglib_opf,
  title={{The Power Grid Library for Benchmarking AC Optimal Power Flow Algorithms}},
  author={Babaeinejadsarookolaee, S. and Birchfield, A. and Coffrin, C. and others},
  journal={arXiv preprint arXiv:1908.02788},
  year={2019}
}

@article{coffrin+vanhentenryck14,
    author = {Coffrin, C. and Van Hentenryck, P.},
    title = {A Linear-Programming Approximation of AC Power Flows},
    journal = {INFORMS Journal on Computing},
    year = {2014},
    volume = {26},
    pages = {718-734}
}

@incollection{knitro,
    author = {Byrd, R.H. and Nocedal, J. and Waltz, R.A.},
    title = {KNITRO: {An} {Integrated} {Package} for {Nonlinear} {Optimization}\
},
    booktitle = {Large-Scale Nonlinear Optimization},
    publisher = {Springer},
    year = {2006},
    editors = {Di Pillo, G. and Roma, M.},
    volume = {83},
    pages = {35-59}
}

@BOOK{molzahn+hiskens19,
  author={Molzahn, Daniel K. and Hiskens, Ian A.},
  title={A Survey of Relaxations and Approximations of the Power Flow Equations},
  year={2019},
  publisher={Now, Foundations and Trends},
  volume={},
  number={},
  pages={},
  doi={}
}

@book{bienstock15,
	title = {Electrical transmission system cascades and vulnerability, an {Operations} {Research} viewpoint},
	language = {en},
        year = {2015},
	author = {Bienstock, D.},
        publisher = {Society for Industrial and Applied Mathematics}
}

@article{cain+etal12,
	title = {History of {Optimal} {Power} {Flow} and {Formulations}},
        journal = {Federal Energy Regulatory Commission},
	language = {en},
	author = {Cain, M. B. and O’Neill, R. P. and Castillo, A.},
	year = {2012},
}

@article{carpentier62,
    author = {Carpentier, J. L.},
    title = {Contribution a l’Etude du Dispatching Economique},
    journal = {Bulletin de la Societe Francoise des Electriciens},
    year = {1962},
    volume = {8},
    issue = {3},
    pages = {431-447}
}

@article{pooling,
    author = {Misener, Ruth and Floudas, Christodoulos},
    year = {2009},
    month = {01},
    pages = {},
    title = {Advances for the pooling problem: Modeling, global optimization, and computational studies Survey},
    volume = {8},
    journal = {Applied and Computational Mathematics}
}

@ARTICLE{beamforming,
  author={Vorobyov, S.A. and Gershman, A.B. and Zhi-Quan Luo},
  journal={IEEE Transactions on Signal Processing}, 
  title={Robust adaptive beamforming using worst-case performance optimization: a solution to the signal mismatch problem}, 
  year={2003},
  volume={51},
  number={2},
  pages={313-324},
  doi={10.1109/TSP.2002.806865}}

@article{bienstock+etal20,
	title = {Mathematical programming formulations for the alternating current optimal power flow problem},
	volume = {18},
	issn = {1619-4500, 1614-2411},
	url = {https://link.springer.com/10.1007/s10288-020-00455-w},
	doi = {10.1007/s10288-020-00455-w},
	number = {3},
	urldate = {2023-10-29},
	journal = {4OR},
	author = {Bienstock, D. and Escobar, M. and Gentile, C. and Liberti, L.},
	month = sep,
	year = {2020},
	pages = {249--292},
}

@misc{amplpy,
  title        = {{AMPL Python API}},
  author       = {{AMPL Optimization Inc.}},
  url          = {https://ampl.com},
  year         = {2026},
  note         = {Accessed: May 23, 2026}
}

@techreport{IEEE738-2012,
  type      = {IEEE Standard},
  title     = {{IEEE Standard for Calculating the Current-Temperature Relationship of Bare Overhead Conductors}},
  author    = {{IEEE Power and Energy Society}},
  number    = {IEEE Std 738-2012},
  year      = {2012},
  month     = {December},
  publisher = {IEEE},
  address   = {New York, NY, USA},
  doi       = {10.1109/IEEESTD.2012.6401981}
}

@article{Kilb26,
title= {{Probabilistic Modeling versus Robust Optimization: A tutorial based on a humanitarian logistics use case}},
author = {Kilb, J and Newman, A. and Bienstock, D.},
pages={1-46},
journal = {	arXiv:2604.02493}, 
year = {2026},
}

@book{kleinrock75,
    author = {Kleinrock, L},
    title = {{Queueing Systems, Volume 1}}, 
     publisher = {John Wiley \& Sons, Inc.},
    year = {1975}
}

@article{kelley60,
    author = {Kelley, J. E. Jr},
    title = {{The Cutting-Plane Method for Solving Convex Programs}},
    journal = {Journal of the Society for Industrial and Applied Mathematics},
    vol = {8},
    pages ={703-712},
    year = {1960}
}

@article{fratta1973flow,
  title={The flow deviation method: An approach to store-and-forward communication network design},
  author={Fratta, L. and Gerla, M. and Kleinrock, L.},
  journal={Networks},
  volume={3},
  number={2},
  pages={97--133},
  year={1973},
  publisher = {Wiley-Blackwell},
  }

@article{ShahrokhiMatula,
  title={The maximum concurrent flow problem},
  author={Farhad Shahrokhi and David W. Matula},
  journal={J. ACM},
  year={1990},
  volume={37},
  pages={318-334},
  url={https://api.semanticscholar.org/CorpusID:4469579}
}

@inproceedings{PlotkinST91,
  author       = {Serge A. Plotkin and David B. Shmoys and {\'{E}}va Tardos},
  title        = {Fast approximation algorithms for fractional packing and covering problems},
  booktitle    = {32nd Annual Symposium on Foundations of Computer Science, San Juan, Puerto Rico, 1-4 October 1991},
  pages        = {495--504},
  publisher    = {{IEEE} Computer Society},
  year         = {1991},
  doi          = {10.1109/SFCS.1991.185411}
}

@article{grigoriadis1994fast,
  title={Fast approximation schemes for convex programs with many blocks and coupling constraints},
  author={Grigoriadis, Michael D and Khachiyan, Leonid G},
  journal={SIAM Journal on Optimization},
  volume={4},
  number={1},
  pages={86--107},
  year={1994},
  publisher={SIAM}
}

@inproceedings{BienstockIyengar2004,
  author    = {Daniel Bienstock and Garud Iyengar},
  title     = {Solving fractional packing problems in {$\mathcal{O}^*(1/\varepsilon)$} iterations},
  booktitle = {Proceedings of the Thirty-Sixth Annual {ACM} Symposium on Theory of Computing},
  series    = {STOC '04},
  pages     = {146--155},
  year      = {2004},
  publisher = {Association for Computing Machinery},
  address   = {New York, NY, USA},
  doi       = {10.1145/1007352.1007376}
}

@article{bienstock2010n,
  title={The {N-K} problem in power grids: New models, formulations, and numerical experiments},
  author={Bienstock, Daniel and Verma, Abhinav},
  journal={SIAM Journal on Optimization},
  volume={20},
  number={5},
  pages={2352--2380},
  year={2010},
  publisher={SIAM}
}

@article{Cormican1998StochasticNI,
  title = {Stochastic Network Interdiction},
  author = {Kelly J. Cormican and David P. Morton and R. Kevin Wood},
  journal = {Operations Research},
  year = {1998},
  volume = {46},
  number = {2},
  pages = {184--197}, 
}

@ARTICLE{swb2,
author = {J. Salmer\'on and K. Wood and R. Baldick},
title = {{Worst-Case Interdiction Analysis of Large-Scale Electric Power Grids}},
journal = {IEEE Trans. Power Systems},
volume = {24},
year = {2009},
institution = {{IEEE}},
pages = {96--104},
}

@article{zeiler2012adadelta,
  title={{ADADELTA}: An Adaptive Learning Rate Method},
  author={Zeiler, Matthew D.},
  journal={arXiv preprint arXiv:1212.5701},
  year={2012}
}

@article{bertsimas2004price,
  title={The Price of Robustness},
  author={Bertsimas, Dimitris and Sim, Melvyn},
  journal={Operations Research},
  volume={52},
  number={1},
  pages={35--53},
  year={2004},
  publisher={INFORMS} 
}

@article{bertsimas2011theory,
  title={Theory and applications of robust optimization},
  author={Bertsimas, Dimitris and Brown, David B and Caramanis, Constantine},
  journal={SIAM Review},
  volume={53},
  number={3},
  pages={464--501},
  year={2011},
  publisher={SIAM}
}

@article{brown2006defending,
  title={Defending critical infrastructure},
  author={Brown, Gerald and Carlyle, Matthew and Salmer{\'o}n, Javier and Wood, Kevin},
  journal={Interfaces},
  volume={36},
  number={6},
  pages={530--544},
  year={2006},
  publisher={INFORMS}
}

@article{qplib,
title = "QPLIB: a library of quadratic programming instances",
author = "Fabio Furini and Emiliano Traversi and Pietro Belotti and Antonio Frangioni and Ambros Gleixner and Nick Gould and Leo Liberti and Andrea Lodi and Ruth Misener and Hans Mittelmann and Sahinidis, {Nikolaos V.} and Stefan Vigerske and Angelika Wiegele",
note = "Publisher Copyright: {\textcopyright} 2018, Springer-Verlag GmbH Germany, part of Springer Nature and The Mathematical Programming Society.",
year = "2019",
month = jun,
day = "1",
doi = "10.1007/s12532-018-0147-4",
language = "English (US)",
volume = "11",
pages = "237--265",
journal = "Mathematical Programming Computation",
issn = "1867-2949",
publisher = "Springer Verlag",
number = "2",
}

@article{benders1962,
  title   = {Partitioning procedures for solving mixed-variables programming problems},
  author  = {Benders, J. F.},
  journal = {Numerische Mathematik},
  volume  = {4},
  number  = {1},
  pages   = {238--252},
  year    = {1962},
  doi     = {10.1007/BF01386316},
  url     = {https://doi.org/10.1007/BF01386316}
}

@book{nocedal2006numerical,
    title = {Numerical Optimization},
    author = {Nocedal, J. and Wright, S.},
    isbn = {9780387400655},
    lccn = {2006923897},
    series = {Springer Series in Operations Research and Financial Engineering},
    year = {2006},
    publisher = {Springer},
    edition = {2} 
}

@book{bienstock2006potential,
  title={Potential Function Methods for Approximately Solving Linear Programming Problems: Theory and Practice},
  author={Bienstock, Daniel},
  volume={53},
  year={2006},
  publisher={Springer Science \& Business Media}
}

@article{blake2,
title = {Polishing solutions to nonlinear optimization problems},
author={Sisson, B. and Bienstock, D.},
year={2026},
journal={forthcoming}
}

@misc{randconcentration,
  author    = {{Martin, B. and Willis, H. and Strong, A.}},
  title     = {{Managing Systemic Supply Chain Risk to the U.S. Economy from Trade Concentration and Geopolitical Conflict: The Roles of Insurance and Other Hedging Strategies} },
  year      = {2026},
  url       ={https://www.rand.org/pubs/research_reports/RRA4291-1.html},
  urldate   = {2026-02-18}
}

@book{schmidt2026,
title = {{Linear and Mixed-Integer Bilevel
Optimization: Theory and Algorithms}},
author={Beck, Y. and Ljubi\'{c}, I and Schmidt, M.},
year = {2026},
publisher = {{Cambridge University Press}},
url={https://yasminebeck.github.io/files/bilevel-optimization-cup.pdf}
}

\end{document}